\documentclass{article}
\usepackage{xcolor}
\usepackage{amsfonts,color}
\usepackage{latexsym,amssymb}
\usepackage{amsmath}
\usepackage[utf8]{inputenc}
\usepackage{dsfont}
\usepackage{enumitem}
\usepackage{float}
\usepackage{graphicx} 
\usepackage{tikz, tkz-euclide}
\usepackage{bbold}
\usepackage{stmaryrd}
\usepackage{babel}
\usepackage{amsthm}
\usepackage[linesnumbered,ruled,vlined]{algorithm2e}
\usepackage{url}
\usepackage{esint}
\usepackage{enumitem}
\usepackage[export]{adjustbox}
\usepackage[sort, numbers]{natbib}

\usetikzlibrary{patterns,patterns,patterns.meta}

\newtheorem{theorem}{Theorem}[section]
\newtheorem{lemma}[theorem]{Lemma}
\newtheorem{proposition}[theorem]{Proposition}
\newtheorem{corollary}{Corollary}
\newtheorem{remark}{Remark}[section]

\newtheorem{definition}[theorem]{Definition}

\usepackage[left=2.7cm,right=2.7cm,top=3.2cm,bottom=3.2cm]{geometry}

\DeclareMathOperator{\Div}{div}

\usepackage[bookmarksnumbered,colorlinks, linkcolor=blue, citecolor=red, bookmarks, breaklinks]{hyperref}

\catcode`@=11 \@addtoreset{equation}{section} \catcode`@=12

\newcommand{\eqdef}{\stackrel{{\rm {def}}}{=}}

\title{Trace and weighted Hardy-Sobolev type inequalities and applications to a quasilinear elliptic problem in half-space}
\author{Loïc CONSTANTIN\footnote{\href{loic.constantin@univ-pau.fr}{loic.constantin@univ-pau.fr}}, Ranieri França FREIRE\footnote{\href{ranieri.franca95@gmail.com}{ranieri.franca95@gmail.com}}, Jacques GIACOMONI\footnote{\href{jacques.giacomoni@univ-pau.fr}{jacques.giacomoni@univ-pau.fr}}}
\date{\small{\it{LMAP (UMR E2S UPPA CNRS 5142) B\^atiment IPRA, Avenue de l'Universit\'{e}, 64013 Pau, France}}}

\renewenvironment{abstract}
{\begin{quote}
\noindent \rule{\linewidth}{.5pt}\par{\bfseries \abstractname.}}
{\medskip\noindent \rule{\linewidth}{.5pt}
\end{quote}
}

\begin{document}

\maketitle
\noindent

\begin{abstract}
In the present paper we  are dealing with the following quasilinear elliptic problem:
\begin{equation*}
		\left\{
		\begin{aligned}
			-\mathrm{div}(\rho(x_N) |\nabla u|^{p-2}\nabla u) &=a|u|^{s-2}u &\mbox{in }&\ \mathds{R}^N_+,
			\vspace{0.pcm}\\
			-|\nabla u|^{p-2}\frac{\partial u}{\partial x_N}&=b|u|^{q-2}u&\mbox{on }&\
			\mathds{R}^{N-1},
		\end{aligned}
		\right.
\end{equation*}
where $a,b\in \mathds{R}$, $p,q,s\in(1,\infty)$ and $\rho$ is a continuous positive function on $[0,+\infty)$.
We first prove  new and sharp embedding results that we establish for the associated weighted energy spaces. In application, we establish existence and regularity of weak solutions to the above problem. We also prove for this problem the nonexistence of nontrivial weak solutions by a new Pohozaev-type identity we obtain. The new results about existence and nonexistence  highlight the role of the weight $\rho$ on the solvability of problem \eqref{PG} 
contrasting strongly with those when $\rho$ is constant.

\end{abstract}

\section{Introduction and main results}

Let $N \geq 3$ and denote by 
$\mathds{R}^N_+ := \{x=(x',x_N) \in \mathds{R}^N : x' \in \mathds{R}^{N-1}, \; x_N > 0\}$ 
the upper half-space, whose boundary is $\partial \mathds{R}^N_+ = \mathds{R}^{N-1}$.
 In this work, we are interested in investigating regularity, existence and nonexistence of nontrivial weak solutions to the quasilinear elliptic problem with nonlinear boundary type conditions:
  \begin{equation}\label{PG}
		\left\{
		\begin{aligned}
			-\mathrm{div}(\rho(x_N) |\nabla u|^{p-2}\nabla u) &=a|u|^{s-2}u &\mbox{in }&\ \mathds{R}^N_+,
			\vspace{0.pcm}\\
			-|\nabla u|^{p-2}\frac{\partial u}{\partial x_N}&=b|u|^{q-2}u&\mbox{on }&\
			\mathds{R}^{N-1},
		\end{aligned}
		\right. 
		\tag{$\mathcal{P}_0$}
	\end{equation}
where $a$ and $b$ are real constants, $s,q \in (1,\infty)$, $p \in (1,N]$. For the weight function $\rho : [0,\infty) \to (0,\infty)$, we assume that $\rho \in C([0,\infty))$ with $\rho(0)=1$ and impose the following growth assumption:
\begin{enumerate}
\item [$(\rho_0)$] There exists constants $0\leq\gamma$ and $C_0>0$ such that 
$$C_0(1+x_N)^\gamma\leq \rho(x_N),\quad \mbox{for all }x_N\geq0.$$ 
\end{enumerate}

If $p<N$, we denote the critical Sobolev exponent by $p^\ast = pN/(N-p)$ and the critical exponent for trace embedding by $p_\ast = p(N-1)/(N-p)$ and by convention $p^*=p_*=\infty$ when $p=N$. The semilinear case (i.e. $p=2$) has been investigated in former contributions (see \cite{AMY,NB,Abreu-JM-Medeiros} for instance). In particular, for the case $\rho$ constant, in \cite{YanYan-Lei} the authors proved that for $s\in [0,2^*]$, $q\leq 2_*$, $s+q<2^*+2_*$, $a\geq 0$, the problem:

 \begin{equation*}
		\left\{
		\begin{aligned}
			-\Delta u &=a|u|^{s-2}u &\mbox{in }&\ \mathds{R}^N_+,
			\vspace{0.pcm}\\
			-\frac{\partial u}{\partial x_N}&=|u|^{q-2}u&\mbox{on }&\
			\mathds{R}^{N-1},
		\end{aligned}
		\right.
\end{equation*}
has no  positive classical solution. For the critical case, in \cite{Escobar2, Li-Zhu} the authors showed that the problem

 \begin{equation*}
		\left\{
		\begin{aligned}
			-\Delta u &=C(N)|u|^{2^*-2}u &\mbox{in }&\ \mathds{R}^N_+,
			\vspace{0.pcm}\\
			-\frac{\partial u}{\partial x_N}&=b|u|^{2_*-2}u&\mbox{on }&\
			\mathds{R}^{N-1},
		\end{aligned}
		\right.
\end{equation*}
admits nontrivial nonnegative solutions only in the form:
$$u(x',x_N)=\frac{\varepsilon^{(N-2)/2}}{(\varepsilon^2+|x'-x_0'|^2+(x_N-t_0)^2)^{(N-2)/2}},$$
for $\varepsilon>0$, $x'_0\in \mathds{R}^{N-1}$ and $-(N-2)^{-1}\varepsilon b=t_0$.
In \cite{NB}, the authors studied \eqref{PG} for $p=2$ and showed that in the case of a nonconstant weight $\rho$ verifying $(\rho_0)$, the existence and nonexistence results are reversed when $\gamma$ is sufficiently large, that is $\gamma>1$, with respect to the case $\rho$ constant. Precisely, the authors got existence in the subcritical case using a mountain pass Theorem and under additional regularity of weak solutions nonexistence in the critical case using a Pohozaev identity. Among the others they showed when $a,b>0$, $\gamma>2$ existence of a weak solution when $q\in(2,2_*)$, $s\in(2,2^*)$ and nonexistence of nontrivial weak solutions in the case $q=2_*$ $a=0$, $b>0$ and in the case $a>0$, $b=0$, $s=2^*$. Contrary to the case $\rho$ constant, they also proved that in the case $\gamma>2$ best constants in subcritical Sobolev and trace embeddings are achieved using subspaces of the energy space with cylindrical symmetry where compactness of embeddings can be revealed.
 
For the quasilinear case, that is  $p\neq 2$, there are very few results available in the literature. In \cite{Nazaret} the author proved a conjecture raised in \cite{Escobar1} that the best constant $Q_n(p)$ of the embedding $\dot{W}^{1,p}(\mathds{R}^{N}_+)\hookrightarrow L^{p_*}(\mathds{R}^{N-1})$ is only attained by the family of extremal functions (parameterized by $\lambda\in \mathds{R}^+$ and $x_N>0$):

$$u_\lambda=\bigg(\frac{\lambda^{2/p}}{(\lambda+x_N)^2+|x'|^2}\bigg)^\frac{N-p}{2(p-1)},$$
and that $u_\lambda$ is the solution of the associated Euler Lagrange equation

\begin{equation*}
		\left\{
		\begin{aligned}
			-\Div(|\nabla u_\lambda|^{p-2}\nabla u_\lambda )&=0 &\mbox{in }&\ \mathds{R}^N_+,
			\vspace{0.pcm}\\
			-|\nabla u_\lambda|^{p-2}\frac{\partial u_\lambda}{\partial x_N}&=Q_n(p) u^{p_*-1}_\lambda&\mbox{on }&\
			\mathds{R}^{N-1}.
		\end{aligned}
		\right.
\end{equation*}
In the subcritical case, in \cite{CAOM} the authors proved existence and nonexistence of weak solutions for the case $\rho$ constant, $a=0$ and under nonlinear Neumann boundary conditions with conflicting type nonlinearities. These results have been extended when $\rho$ satisfies $(\rho_0)$ with $\gamma>p-1$ in \cite{doO_Freire_Medeiros_2026}. Concerning the case $a>0$, Liouville type theorems implying nonexistence of nonnegative weak solutions under the condition $s<s_c$ where $s_c<p^*$ have been established in former contributions \cite{Damascelli,chen,Birindelli}. 
The case $\rho$, again verifying $(\rho_0)$ with $\gamma>p-1$, and $b=0$ has been considered in \cite{Do-Freire-Medeiros} where Liouville type results are established.
In the present paper in the continuation of \cite{NB, Do-Freire-Medeiros, doO_Freire_Medeiros_2026} we are further investigating the quasilinear case with weight satisfying $(\rho_0)$ with the natural question in the backdrop: {\it What is the borderline condition on $\gamma$ for assuring reversal existence and nonexistence results with respect to the case $\rho$ constant?} We highlight that by using the  weighted Hardy–Sobolev inequality \eqref{Hardy p}, former contributions proved these results by requiring $\gamma>p-1$ or $\gamma>p$. In the present paper, we show that these conditions on $\gamma$ are not optimal and establish more precisely the effect of the growth of $\rho$ in existence and nonexistence issues with respect to the constant case. Roughly speaking, we get existence of nontrivial solutions in the subcritical case whereas nonexistence is proved in the critical case contrasting with the case where $\rho$ is constant. Answering the above question, we further show that condition $\gamma>0$ is sufficient to obtain these results, up to some restrictions on $s$ and $q$. These new results improve also significantly the semilinear case done in \cite{NB} where $\gamma>1$ or $\gamma\geq 2$ are required. Since \eqref{Hardy p} is no longer exploitable for $\gamma>0$ small (in this regard, see Corollary 1.10 and Remark 1.11 in \cite{WeightedIneq}), to obtain the requested improvement, we need to establish in $\mathcal{D}^{1,p}_\rho(\mathds{R}^N_+)$ new  Hardy-Sobolev and trace embeddings that are of independent interest (see Theorem 1.1 below) and that extend former results (see for instance Theorems 1.4 and 1.6 in \cite{Abreu-Furtado-Medeiros2}). In contrast with former contributions, we consider in addition the case $p=N$  for which we also establish new Sobolev and trace embeddings. We also obtain a Trudinger Moser type inequality providing embeddings of $\mathcal{D}^{1,p}_\rho(\mathds{R}^N_+)$ in Orlicz spaces with exponential growth (see Theorem \ref{TruMos}). 

To start with, let us recall that the trace Theorem in the half-space implies the existence of a positive constant $C(N,p)$ such that, for all functions $u \in C^\infty_0(\mathds{R}^N)$ and $p<N$, the following inequality holds:
\begin{equation}\label{Trace-inequality}
    C(N,p) \left( \int_{\mathds{R}^{N-1}} |u|^{p_{\ast}}\, \mathrm{d}x' \right)^{p/p_{\ast}} 
    \leq \int_{\mathds{R}^N_+} |\nabla u|^p\, \mathrm{d}x
\end{equation}
whereas the classical Sobolev inequality leads the existence of $S(N,p)>0$ such that :
 \begin{equation}\label{GNS}
S(N,p)\left(\int_{\mathds{R}^N_+}|u|^{p^\ast}\, \mathrm{d} x\right)^{p/p^\ast}\leq \int_{\mathds{R}^N_+}|\nabla u|^p\, \mathrm{d} x,~\forall\, u \in C^{\infty}_0(\mathds{R}^N).
\end{equation}
In case of weighted diffusion, Hardy-Sobolev type inequalities are required and have been consequently studied in many and various references. Firstly, one can quote the seminal work of Maz'ya \cite{mazya} where the well-known Hardy-Sobolev-Maz'ya inequality is established:
$$\int_{\mathds{R}_+^N} |\nabla u|^2\mathrm{d}x-\frac{1}{4}\int_{\mathds{R}_+^N}\frac{u^2}{x_N^2}\mathrm{d}x\geq C\Big(\int_{\mathds{R}_+^N}x_N^\gamma |u|^{2_\gamma}\mathrm{d}x\Big)^{2/{2_\gamma}},\; \forall\, u\in C^\infty_0({\mathds{R}_+^N}),$$
with $2<2_\gamma\leq 2^*$ and $\gamma=(N-2)\frac{2_\gamma}{2}-N$. In a series of papers, Filippas, Maz’ya and Tertikas study similar inequalities on convex domains (see \cite{Fi-Ma-Te1,Fi-Ma-Te2}). In \cite{Ambrosio}, dealing with weighted quasilinear type operators, the authors establish the following Hardy-Sobolev type inequality:
\begin{equation}\label{hardy}
\left| \frac{\gamma - p + 1}{p} \right|^p
\int_{\mathds{R}^N_+} \frac{|u|^p}{x_N^{p - \gamma}} \, \mathrm{d}x
\le
\int_{\mathds{R}^N_+} x_N^\gamma |\nabla u|^p \, \mathrm{d}x,
\quad \forall\, u \in C_0^\infty(\mathds{R}^N_+),
\end{equation}
for all $\gamma\in \mathds{R}$. In case of non homogeneous Dirichlet or Neumann type boundary conditions, the above Hardy-Sobolev type inequalities can not be used and some other types of Hardy inequalities have been proved more recently. For instance in \cite{Abreu-Furtado-Medeiros2} (see also \cite[Theorem 1]{doO_Freire_Medeiros_2026}), the following weighted Hardy–Sobolev-type inequality is established: Let 
$N\geq 2$ and $0<p-1<\gamma$, then for every $u \in C^\infty_0(\mathds{R}^N)$ it holds:
\begin{equation}\label{Hardy p}
C_{p,\gamma}^p \int_{\mathds{R}_{+}^N}\frac{|u|^p}{(1+x_N)^{p-\gamma}}\, \mathrm{d}x + C_{p,\gamma}^{p-1} \int_{\mathds{R}^{N-1}}|u|^p\, \mathrm{d}  x'\leq \int_{\mathds{R}_{+}^N}(1+x_N)^\gamma|\nabla u|^p \, \mathrm{d} x,
\end{equation}
with sharp constant $C_{p,\gamma}:=\frac{\gamma-p+1}{p}$. In particular this inequality leads  trace and Sobolev embeddings in the case $\gamma>p-1$ and $\gamma\geq p$ respectively. In case $\gamma<p-1$, the authors proved:
\begin{equation}\label{HardyDec}
\left( \frac{p - \gamma - 1}{p} \right)^p \int_{\mathds{R}^N_+} \frac{|u|^p}{(1 + x_N)^{p-\gamma}} \mathrm{d}x \le  \int_{\mathds{R}^N_+} (1 + x_N)^\gamma |\nabla u|^p \mathrm{d}x  + \left( \frac{p - \gamma - 1}{p} \right)^{p-1} \int_{\mathds{R}^{N-1}} |u|^p \mathrm{d}x'
\end{equation}
from which one can see that the embedding $\mathcal{D}^{1,p}_\rho(\mathds{R}^N_+)$ in $L^p(\mathds{R}^{N-1})$ is not outcoming anymore. In \cite{WeightedIneq}, the authors study more general weighted Hardy–Sobolev inequalities and got: 
$$\int_{\Omega} W' |u|^p \, \mathrm{d}x + p \int_{\mathds{R}^{N-1}} W(x', \psi(x')) |u(x', \psi(x'))|^p \, \mathrm{d}x' \le p^p \int_{\Omega} \frac{W^p}{|W'|^{p-1}} |\nabla u|^p \, \mathrm{d}x,$$
where $W$ is an increasing and differentiable weight. In this matter, the case $\gamma<p-1$ corresponds to a decreasing weight of the form $W(x_N)=(1+x_N)^{\gamma-(p-1)}$ and the authors established an inequality generalizing \eqref{HardyDec} (see Theorem 1.12 there). In \cite{WeightedIneq}, the authors further show that when $0\leq \gamma\leq p-1$ the inequality \eqref{hardy} 
does not hold for non homogeneous trace functions,  that is for the restriction of $u \in C^\infty_0(\mathds{R}^N)$ in the upper half-space.
In this context, we combine \eqref{Hardy p} and the classical Sobolev inequalities (\ref{Trace-inequality}), (\ref{GNS}) in a nontrivial way to establish for all $\gamma>0$, $p\leq N$, inequality \eqref{ineqtheo1} given in Theorem \ref{theo1}.
Thanks to the associated Sobolev embedding, we further get a Trudinger Moser inequality for the case $p=N$ and any $\gamma>0$ as given in Theorem \ref{TruMos} and extending Theorem 1.5 in \cite{Abreu-Furtado-Medeiros2}.

Using the compact embeddings obtained in Theorem \ref{compa1} stated below that generalizes Theorem 1.2 in \cite{NB} and mountain pass Theorem, we get existence of nonnegative and nontrivial weak solutions for $\rho$ verifying $(\rho_0)$ with $\gamma>0$ and $s,q$ belonging to the intervals $(s_\gamma,p^*)$ and $(q_\gamma ,p_*)$ respectively. Another striking consequences of the compactness embeddings are that best constants in Sobolev and trace embeddings are achieved with no more restrictions than $\gamma>0$ as stated in Theorem \ref{bestcons}.

Next, we establish H\"older and Sobolev regularity results of weak solutions to \eqref{PG} that we use subsequently to derive 
a new Pohozaev type identity stated in Theorem \ref{Pohozaev}. In addition, in the critical case setting, that is $s=p^*$ or $q=p_*$, using Theorem \ref{Pohozaev} we get nonexistence of non trivial weak solutions to \eqref{PG} when $\gamma>0$ (see Theorems \ref{Nonexistence1} and \ref{Nonexistence2}). Again, these results improve those in \cite{NB} where assumptions $\gamma>1$ and $p=2$ are assumed.

Before stating our main results, we introduce the functional framework we are using throughout the paper: Let $C^\infty_\delta(\mathds{R}^N_+)$ be the set of the functions in $C^\infty_0(\mathds{R}^N)$ restricted to $\mathds{R}^N_+$ and consider the weighted Sobolev space $\mathcal{D}^{1,p}_\rho(\mathds{R}^N_+)$ defined as the closure of $C_\delta^\infty(\mathds{R}^N_+)$ with respect to the norm
$$\|u\|:=\left(\int_{\mathds{R}^{N}_+}\rho(x_N)|\nabla u|^p\, \mathrm{d} x\right)^{1/p}.$$
We consider also the following Banach spaces 
\begin{equation*}
    \begin{aligned}
        E_s(\mathds{R}^N_+) :=\, & \overline{C^\infty_\delta(\mathds{R}^N_+)}^{\|\cdot\|_{E_s(\mathds{R}^N_+)}} \text{ and }  E_q(\mathds{R}^{N-1}) :=\, & \overline{C^\infty_\delta(\mathds{R}^N_+)}^{\|\cdot\|_{E_q(\mathds{R}^{N-1})}},
    \end{aligned}
\end{equation*}
for $s,q\in(1,\infty)$, where 
\begin{equation*}   
\|u\|_{E_p(\mathds{R}^N_+)}=\|u\|+\|u\|_{L^p(\mathds{R}^N_+)}\quad\text{and}\quad  \|u\|_{E_q(\mathds{R}^{N-1})}=\|u\|+\|u\|_{L^q(\mathds{R}^{N-1})}.
\end{equation*}
By Proposition \ref{embedhardy} as a consequence of \eqref{Hardy p}, $E_s(\mathds{R}^N_+)=\mathcal{D}^{1,p}_\rho(\mathds{R}^N_+)$ when $\gamma\geq p$ and $s\in[p,p^\ast]$ and similarly, $E_q(\mathds{R}^{N-1})=\mathcal{D}^{1,p}_\rho(\mathds{R}^N_+)$ when $\gamma>p-1$ and $q\in[p,p_\ast]$. We also denote $\mathcal{R}$ the subspace of $\mathcal{D}^{1,p}_\rho(\mathds{R}^N_+)$ composed of functions $u$ that are radially symmetric with respect to  $x'$, that is $u\in \mathcal{D}^{1,p}_\rho(\mathds{R}^N_+)$ such that $u(x',x_N)=v(|x'|,x_N)$ for some $v$. We also define $\mathcal{R}_s=\mathcal{R}\cap L^s(\mathds{R}^N_+)$, $\mathcal{R}_q=\mathcal{R}\cap L^q(\mathds{R}^{N-1})$ endowed with the usual intersection norms. We define the exponents:
\[
q_\gamma =
\begin{cases}
\frac{p(N-1)}{N-p+\gamma} & \text{if }0< \gamma \leq  p-1, \\
p   & \text{if } \gamma>p-1,
\end{cases}
\quad
s_\gamma =
\begin{cases}
\frac{Np}{N-p+ \gamma}& \text{if }\gamma\in (0, p), \\
p   & \text{if } \gamma\geq p.
\end{cases}
\]

In order to consider the case $0<\gamma\leq p-1$, we prove  the following 
 Hardy type inequality of independent interest,  extending \eqref{Hardy p} and also \cite[Theorem 2.12]{Do-Freire-Medeiros} (for $\gamma> p-1$):

\begin{theorem}\label{theo1}
Let $\gamma>0$, $1<p\leq  N$ then for all $s\in [s_\gamma,p^*)$, $q\in (q_\gamma,p_*)$, there exists $C>0,$
$$K(s)\eqdef \begin{cases}
\frac{(N-p)(s(N-p+\gamma)-Np)}{p(Np-N\gamma-p)}, &\text{if } \gamma\in(0, p-1],\\
s/p(N-p+\gamma)-N, & \text{if } \gamma\in(p-1, \infty),\\
\end{cases}$$
such that for all $k<K(s)$ (even $k\leq K(s)$ if $\gamma>p-1$ or if $s=s_\gamma$), $u\in C^\infty_\delta(\mathds{R}^N_+)$, one has:
\begin{equation}\label{ineqtheo1}
\Big( \int_{\mathds{R}^N_+}(1+x_N)^{k}|u|^s\, \mathrm{d}x \Big)^{ \frac{p}{s} }+\Big( \int_{\mathds{R}^{N-1}} |u|^q\, \mathrm{d}x^\prime \Big)^{ \frac{p}{q} }\leq C\int_{\mathds{R}^N_+} (1+x_N)^\gamma|\nabla u|^p\, \mathrm{d}x.
\end{equation}
\end{theorem}
Assuming $(\rho_0)$ with $\gamma>0$, $1<p\leq N$, $s\in [s_\gamma,p^*)$, $q\in (q_\gamma,p_*)$, we deduce from Theorem \ref{theo1} the continuous embeddings:
$$\mathcal{D}^{1,p}_\rho(\mathds{R}^N_+)\hookrightarrow L^q(\mathds{R}^{N-1}), \; \mathcal{D}^{1,p}_\rho(\mathds{R}^N_+)\hookrightarrow L^s(\mathds{R}^{N}_+).$$
Remark \ref{est-best const} further gives sharp estimates on best constants associated to these embeddings. Next, using the Sobolev embedding in the dimensional case $p=N$ we get the following Moser Trudinger inequality:
\begin{theorem}\label{TruMos}
Let $N\geq2$, $\gamma>0$,  $0\leq \beta<N$, $s> s_\gamma$ and let $q > s\!\left(1 - \frac{\beta}{N}\right)$ ( $q \geq s$ if $\beta = 0 $ ). Then for all $0\leq \alpha<2^{-\frac{1}{N-1}}\alpha_N $, 
with $\alpha_N$ defined in \eqref{alphaN}, we have
$$\sup_{\{u\in \mathcal{D}^{1,N}_\rho(\mathds{R}^N_+), \|u\|\leq 1\}}\int_{\mathds{R}^N_+} \frac{\exp\!\left(\alpha\!\left(1 - \frac{\beta}{N}\right)|u|^{\frac{N}{N-1}}\right)\,|u|^q}{|x|^\beta}\, \mathrm{d}x <\infty.$$
\end{theorem}
In frame of space $\mathcal{R}$ with cylindrical symmetry, we next prove that
\begin{theorem}\label{compa1}
Let $1<p\leq N$ and assume the condition $(\rho_0)$ with $\gamma>0$, and let $s\in(s_\gamma,p^*)$, $q\in (q_\gamma,p_*)$, we have that the continuous embeddings: 
$$\mathcal{R}\hookrightarrow L^q(\mathds{R}^{N-1}), \; \mathcal{R}\hookrightarrow L^s(\mathds{R}^{N}_+)$$
are compacts.
\end{theorem}
An important consequence of Theorem \ref{theo1} and Theorem \ref{compa1}, which contrasts with the case $\rho$ constant is the existence of extremal functions for best constants in associated Sobolev and trace embeddings:
\begin{theorem}\label{bestcons}
Let $\gamma>0,$ $1<p\leq N$, $q\in (q_\gamma,p_*) $, the following best constant is attained:
$$S_q=\inf_{\mathcal{D}^{1,p}_\rho(\mathds{R}^N_+)\backslash\{0\}} \frac{\int_{\mathds{R}^N_+} (1+x_N)^\gamma|\nabla u |^p\, \mathrm{d}x}{(\int_{\mathds{R}^{N-1}}|u|^q\, \mathrm{d}x')^\frac{p}{q}}$$
and for $s\in(s_\gamma,p^*)$ the following best constant is achieved:
$$S_s=\inf_{\mathcal{D}^{1,p}_\rho(\mathds{R}^N_+)\backslash\{0\}} \frac{\int_{\mathds{R}^N_+} (1+x_N)^\gamma|\nabla u |^p\, \mathrm{d}x}{(\int_{\mathds{R}^{N}_+}|u|^s\, \mathrm{d}x)^\frac{p}{s}}.$$
\end{theorem}
\begin{remark}\label{remarque optimalité}
Similarly as in Theorem 1.3 in \cite{NB}, we can show that in case $q=p$ or $s=p$ the constant is not achieved. Observe also that as a consequence of Lemma 2.1 in \cite{Abreu-Furtado-Medeiros2}, the constants $S_q$ and $S_s$ are equal $0$ when $\gamma\leq p-1$ and $q=s=p$ whereas they are positive in case of $\gamma>p-1$ as a consequence of \eqref{Hardy p}. Due to classical trace and Sobolev embeddings and from \eqref{Hardy p}, the constants $S_{p_*}$ and $S_{p^*}$ are also positive but not achieved due to regularity and nonexistence results (see Theorem \ref{Nonexistence1}). These observations highlight that Theorem \ref{bestcons} is sharp.
\end{remark}
 
The proof of Theorem \ref{bestcons} is given in Section \ref{secEmb}. We now define the notion of a weak solution that we adopt in the present paper: 

\begin{definition}\label{weak solution}
We call a weak solution of \eqref{PG} a function $u\in \mathcal{D}^{1,p}_\rho(\mathds{R}^N_+)$ such that:
\begin{equation}\label{varFor}
\int_{\mathds{R}^N_+} \rho(x_N) |\nabla u|^{p-2}\nabla u.\nabla \varphi \, \mathrm{d}x =a\int_{\mathds{R}^N_+}|u|^{s-2}u\varphi \, \mathrm{d}x +b\int_{\mathds{R}^{N-1}}|u|^{q-2}u\varphi\, \mathrm{d}x^\prime, 
\end{equation}
for every function $\varphi\in C^\infty_\delta(\mathds{R}^N_+)$.
\end{definition}
In order to, recover the boundary conditions, we establish higher order regularity for nonnegative weak solutions of \eqref{PG}. The proofs of these regularity results are given in section \ref{secReg}. The first result concerns the H\"older regularity:

\begin{theorem}\label{C1alpha regularity}
Let $p\leq N$. Assume that condition $(\rho_0)$ holds together with $\rho \in C^{0,\beta}_{\mathrm{loc}}[0,\infty)$ for some $\beta\in (0,1)$. 
If $u$ is a weak solution, then $u \in C^{1,\alpha}_{\mathrm{loc}}(\overline{\mathds{R}^N_+})\cap L^\infty(\mathds{R}^N_+)$ for some $\alpha \in (0, 1)$ and $u$ verifies the boundary condition in \eqref{PG} pointwisely, 
in each of the following cases:
\begin{itemize}

\item[$(i)$] $u \in E_q(\mathds{R}^{N-1})$, with $a > 0$, $b \leq 0$, $q \in (1,\infty)$ and $\gamma>0$ with $s \in (s_\gamma,p^\ast)$ or alternatively $\gamma\geq p$ with $s \in [p,p^\ast)$;
\item[$(ii)$] $u \in E_q(\mathds{R}^{N-1})$, with $\gamma \ge 0$, $p<N$, $a > 0$, $b \leq 0$, $s = p^\ast$, and $q \in (1,\infty)$;
\item[$(iii)$] $u \in E_s(\mathds{R}^N_+)$, with $a \leq 0$, $b > 0$, $s \in  (1,\infty)$ and $\gamma>0$ with  $q \in (q_\gamma, p_\ast)$ or alternatively $\gamma>p-1$ with $q \in [p, p_\ast)$;
\item[$(iv)$] $u \in E_s(\mathds{R}^N_+)$, with $\gamma \ge 0$, $p<N$, $a \leq 0$, $b > 0$, $s \in (1,\infty)$, and $q = p_\ast$.
\end{itemize}
\end{theorem}
A key point in the proof of Theorem \ref{C1alpha regularity} is the uniform bound obtained by Moser type iterations technique, thanks to trace and Sobolev embeddings in Theorem \ref{theo1}. Next, we are concerned with Sobolev regularity. Precisely, using the quotient difference method, we establish:

\begin{theorem}\label{W2regularity}
Let $p\leq N$. Assume that condition $(\rho_0)$ holds together with $\rho \in W^{1,\infty}_{\mathrm{loc}}(0,\infty)$.
Let $u$ be a weak solution of \eqref{PG}. 
Then, in each of the following cases 
\begin{itemize}

\item[$(i)$] $u \in E_q(\mathds{R}^{N-1})$, with $a > 0$, $b \leq 0$, $q \in (1,\infty)$ and $\gamma>0$ with $s \in (s_\gamma,p^\ast)$ or alternatively $\gamma\geq p$ with $s \in [p,p^\ast)$;
\item[$(ii)$] $u \in E_q(\mathds{R}^{N-1})$, with $\gamma \ge 0$, $p<N$, $a > 0$, $b \leq 0$, $s = p^\ast$, and $q \in (1,\infty)$;
\item[$(iii)$] $u \in E_s(\mathds{R}^N_+)$, with $a \leq 0$, $b > 0$, $s \in  (1,\infty)$ and $\gamma>0$ with  $q \in (q_\gamma, p_\ast)$ or alternatively $\gamma>p-1$ with $q \in [p, p_\ast)$;
\item[$(iv)$] $u \in E_s(\mathds{R}^N_+)$, with $\gamma \ge 0$, $p<N$, $a \leq 0$, $b > 0$, $s \in (1,\infty)$, and $q = p_\ast$;
\end{itemize}
the following regularity result holds:
\begin{enumerate}
\item  If $1 < p \leq 2$, then $u \in W^{2,p}_{\mathrm{loc}}(\mathds{R}^N_+)$;
\item  If $p > 2$, then $u \in W^{2,2}_{\mathrm{loc}}(U)$, where $U:=\{x \in \mathds{R}^N_+: \nabla u(x)\neq 0\}.$
\end{enumerate}
\end{theorem}
Concerning the general case $a,b>0$, we show the following local regularity result:
\begin{theorem}\label{regularity a,b>0}
Let $p\leq N$. Assume that condition $(\rho_0)$ holds together with $\rho\in W^{1,\infty}_{\mathrm{loc}}(0,\infty)$ and that $a,b>0$. 
Let $u$ be a  weak solution of \eqref{PG}. 
Then, in each of the following cases 
\begin{itemize}
\item[$(i)$] $s \in [p,p^\ast]$, $q \in [p,p_\ast]$ with $p<N$ ($s\in [p,\infty)$, $q\in [p,\infty)$ if $p=N$) and $\gamma\geq p$;
\item[$(ii)$] $s \in (s_\gamma,p^\ast]$, $q \in [p,p_\ast]$ with $p<N$ ($s\in (s_\gamma,\infty)$, $q\in [p,\infty)$ if $p=N$) and $\gamma>p-1$;
\item[$(iii)$] $s \in (s_\gamma,p^\ast]$, $q \in (q_\gamma, p_\ast]$ with $p<N$ ($s\in (s_\gamma,\infty)$, $q\in (q_\gamma,\infty)$ if $p=N$) and $\gamma>0$;
\end{itemize}
we have $u\in C^{1,\alpha}_{\mathrm{loc}}(\mathds{R}^N_+)$ for some $\alpha \in (0,1)$. Moreover,
\begin{itemize}
\item [1.] If $1 < p \leq 2$, then $u \in W^{2,p}_{\mathrm{loc}}(\mathds{R}^N_+)$;
\item [2.] If $p > 2$, then $u \in W^{2,2}_{\mathrm{loc}}(U)$, where $U=\{x \in \mathds{R}^N_+: \nabla u(x)\neq 0\}$.
\end{itemize}
\end{theorem}

In section \ref{secReg}, following the approach in \cite{Ilyasov-Takac}, we also prove additional regularity of the nonlinear term $|\nabla u|^{p-1}$ (see Lemma \ref{Lou}). Paired with the $W^2_{\mathrm{loc}}$-regularity, this result provides Pohozaev identity and consequently nonexistence results, stated below and proved in Sections \ref{secPoho} and \ref{secNon} respectively. We now state existence results proved in Section \ref{secExi}:

\begin{theorem}\label{th1}
Let $(\rho_0)$ holds with $\gamma>0$ and $1<p\leq N$. Then, assuming one of the following cases:

\begin{itemize}
    \item[(i)] $a\leq 0$, $b>0$, $s\in(1,p_*)$, $q \in (\max(q_\gamma,s),p_*)$,
    
    \item [(ii)] $a>0$, $b\leq 0$, $q\in(1,p^*)$, $s\in (\max (s_\gamma,q),p^* ),$
\end{itemize}
there exists $u\in \mathcal{D}^{1,p}_\rho(\mathds{R}^N_+)\backslash\{0\}$,
a nonnegative weak solution to \eqref{PG}. Furthermore, $u\in  \mathcal{R}_s$ (resp. $u\in  \mathcal{R}_q$) if (i) holds (resp. (ii) holds). Moreover if $\rho \in W^{1,\infty}_{\mathrm{loc}}[0,\infty)$ then $u \in C^{1,\alpha}_{\mathrm{loc}}(\overline{\mathds{R}^N_+})\cap L^\infty(\mathds{R}^N_+)$ for some $\alpha \in (0,1)$, $u>0$ in $\overline{\mathds{R}^N_+}$ and verifies the boundary condition pointwisely. Furthermore, the regularity results of Theorem \ref{W2regularity} holds.

\end{theorem}

In the case $a\leq 0$ and $b\leq 0$ we use the framework of $E_s(\mathds{R}^N_+)$ and $E_q(\mathds{R}^{N-1})$ respectively to establish Theorem \ref{th1}. 
While when $a,b>0$ we show:

\begin{theorem}\label{th3}
Assume $(\rho_0)$ with $\gamma>0$, $a>0$, $b>0$ and $1<p\leq N$. If $q\in (q_\gamma ,p_*)$, $s\in (s_\gamma ,p^*)$ then there exists a nontrivial nonnegative weak solution of \eqref{PG}, belonging to $\mathcal{R}$. Moreover if $\gamma>p$ and if one of the following holds
\begin{itemize}
\item $q=p$, $s\in (p,p^*)$, $b/C_0<C_{p,\gamma}^{p-1}$

\item $s=p$, $q \in (p,p_*)$, $a/C_0<C_{p,\gamma}^{p}$,
\end{itemize}
where $C_{p,\gamma}=\frac{\gamma-p+1}{p}$ and $C_0$ is given in $(\rho_0)$, then there exists a nontrivial nonnegative weak solution of \eqref{PG}, belonging to $\mathcal{R}$. Furthermore, if $\rho\in W^{1,\infty}_{\mathrm{loc}}(0,\infty)$, $u>0$ in $\mathds{R}^N_+$ and satisfies regularity results of Theorem \ref{regularity a,b>0}.

\end{theorem}
We now state the Pohozaev identity we establish for solutions to problem with more general weights and source terms. This identity is the core of the proof of nonexistence of nontrivial weak solutions to \eqref{PG} and is proved in Section \ref{secPoho}.

\begin{theorem}[Pohozaev identity]\label{Pohozaev} Let $f,g:\mathds{R}\longrightarrow\mathds{R}$ be continuous functions. Assume that $u$ satisfies regularity results of Theorem \ref{W2regularity} and that $u\,\in   C^{1}(\overline{\mathds{R}^N_+})$, is a weak solution of the problem
		\begin{align*}
			\left\{
			\begin{aligned}
				-\mathrm{div}(\rho(x)|\nabla u|^{p-2} \nabla u)&=f(u)&\quad&\text{ in }\quad\mathds{R}_{+}^{N},\\
				\rho(x^\prime,0)|\nabla u|^{p-2}\frac{\partial u}{\partial \nu}&=g(u)&\quad&\text{ on }\quad\mathds{R}^{N-1}.
			\end{aligned}
			\right.
		\end{align*}
		We suppose also that $|\nabla u|^r \in W^{1,1}_\mathrm{loc}(\mathds{R}^N_+)$ for some $r\in (1,p)$, $\rho(x)|\nabla u|^p\in L^1(\mathds{R}^N_+)$, where $\rho \in C^{1,\alpha}_{\mathrm{loc}}(\overline{\mathds{R}^N_+})$ for some $\alpha\in (0,1)$ and satisfies additionally for some $c>0$
        \begin{equation}\label{rho1}
            0<|\langle \nabla \rho(x), x\rangle|\leq c\rho(x),\quad \forall\, x\neq 0,
        \end{equation}
        $F(u)\in L^1(\mathds{R}^N_+)$, and $G(u)\in L^1(\mathds{R}^{N-1})$  where 
		\[
		F(t)=\int_0^tf(z)\mathrm{d}z, \ \ \mbox{ and }\ \ G(t)=\int_0^tg(z)\mathrm{d}z.
		\]
		Then, $u$ satisfies the following Pohozaev type identity
		\[
		\frac{N-p}{p}\int_{\mathds{R}_{+}^{N}}\rho(x)|\nabla u|^p\mathrm{d}x+\frac{1}{p}\int_{\mathds{R}_{+}^{N}}\langle \nabla \rho(x) , x\rangle |\nabla u|^p\mathrm{d}x=N\int_{\mathds{R}^N_+}F(u)\mathrm{d}x+(N-1)\int_{\mathds{R}^{N-1}}G(u)\mathrm{d}x^\prime.
		\]
	\end{theorem}
For applying Theorem \ref{Pohozaev} in our setting, we assume the following condition on $\rho$ (that implies \eqref{rho1}):
\begin{enumerate}
\item [$(\rho_1)$]$\rho \in\, C^{1,\alpha}_{\mathrm{loc}}[0,\infty)$ for some $\alpha\in (0,1)$ and there exists a constant $c_1>0$ such that 
$$0< \rho^\prime(s)s\leq c_1 \rho(s), \quad \forall \, s>0.$$
\end{enumerate}
We now state nonexistence results proved in Section \ref{secNon}. Using regularity results of Theorems \ref{C1alpha regularity} and \ref{W2regularity} together with Theorem \ref{Pohozaev}, we establish first
\begin{theorem}\label{Nonexistence1}
Assume $(\rho_0)$ and $(\rho_1)$ with $\gamma>0$ and $p<N$. Then $u\equiv 0$ in the following cases:
\begin{itemize}
\item[$(i)$] $u$ is a weak solution in $E_q(\mathds{R}^{N-1})$, with $a > 0$, $b \leq 0$, $s = p^\ast$, and $q \in (1, p_\ast]$;
\item[$(ii)$] $u$ is a weak solution in $E_s(\mathds{R}^N_+)$, with $a \leq 0$, $b > 0$, $s \in (1,p^\ast]$, and $q = p_\ast$.
\end{itemize}
\end{theorem}
Next, assuming regularity of weak solutions,  the following general nonexistence results hold:
\begin{theorem}\label{Nonexistence2} Assume $(\rho_0)$ with $\gamma>0$ and $(\rho_1)$ and let $p<N$. Let $u\in \mathcal{D}^{1,p}_\rho(\mathds{R}^N_+)\cap C^{1}(\overline{\mathds{R}^N_+})\cap L^\infty(\mathds{R}^N_+)$, be a 
weak solution of \eqref{PG} satisfying Sobolev regularity results in Theorem \ref{W2regularity} and such that $|\nabla u|^r \in W^{1,1}_{\mathrm{loc}}(\mathds{R}^N_+)$ for some $r\in (1,p)$, when $p>2$. 
Then $u\equiv 0$ in the following cases:
\begin{itemize}
\item[$(i)$] $u\in E_q(\mathds{R}^{N-1})$, $a > 0$, $b \leq 0$, and $s \geq p^\ast$ and $q\in(1,p_\ast]$;
\item[$(ii)$] $u\in E_s(\mathds{R}^N_+)$, $a \leq 0$, $b > 0$, $s \in (1,p^\ast]$ and $q \geq p_\ast$;
\item [$(iii)$] $a>0$, $b>0$, $s\geq p^\ast$ and $q\geq p_\ast$.
\end{itemize}
\end{theorem}

Concerning the case $p\geq N$, we can also derive nonexistence results:
\begin{theorem}\label{nonexPleqN} Let $\rho$ verifying $\eqref{rho1}$ with $\rho'(x_N)< 0$ and $p\geq N$. Let $u\in \mathcal{D}^{1,N}_\rho(\mathds{R}^N_+)\cap C^{1}(\overline{\mathds{R}^N_+})\cap L^\infty(\mathds{R}^N_+)$, be a 
weak solution of \eqref{PG} satisfying Sobolev regularity results in Theorem \ref{W2regularity} and such that {$|\nabla u|^r \in W^{1,1}_{\mathrm{loc}}(\mathds{R}^N_+)$} for some $r\in (1,N)$, when $N>2$. Let $a,b>0$, $q,s>1$, if $u\in E_q(\mathds{R}^{N-1})\cap  E_s(\mathds{R}^N_+)$ then $u\equiv0$.
\end{theorem}


\section{New Sobolev and trace embeddings}\label{secEmb}

In this section, we give the proof of embedding results stated in the introduction, in particular Theorems \ref{theo1} and \ref{TruMos} implying new trace and Sobolev embeddings of the energy space $\mathcal{D}^{1,p}_\rho(\mathds{R}^N_+)$. We also establish compact embedding results given in Theorem \ref{compa} from which the proof of Theorem \ref{compa1} follows. These results are used subsequently for showing existence of nontrivial weak solutions, existence of extremal functions for best constants and regularity results. We start with Theorems \ref{inj2} and \ref{improvetheo2.2} given below. For sake of clarity, we work with $\rho(x_N)=(1+x_N)^\gamma$. Obviously these embeddings hold also for $\mathcal{D}^{1,p}_\rho(\mathds{R}^N_+)$ when $\rho$ satisfies $(\rho_0)$. We first consider the case $p\in (1,N)$. 
From \eqref{Trace-inequality} and the interpolation of \eqref{Hardy p} with \eqref{GNS}, it follows:

\begin{proposition}\label{embedhardy}
For all $p\in(1,N)$, $\gamma> p-1$, for all $q\in [p,p_*]$ we have $\mathcal{D}^{1,p}_\rho(\mathds{R}^N_+)\hookrightarrow  L^q(\mathds{R}^{N-1})$.
In addition, for $\gamma \geq p$, we have $\mathcal{D}^{1,p}_\rho(\mathds{R}^N_+)\hookrightarrow  L^s(\mathds{R}^{N}_+)$ for all $s\in [p,p^*]$.
\end{proposition}

Using a different proof, we show a new and sharp Sobolev embedding that concerns the case $\gamma<p$. First, we have the following result:


\begin{theorem}\label{inj2}
Let $p<N$ and $\gamma>0$.


Then, the following statements hold:

\begin{itemize}
\item[(i)] $\mathcal{D}^{1,p}_\rho(\mathds{R}^N_+)\hookrightarrow L^s(\mathds{R}^N_+)$ for all $s\in [s_\gamma,p^*]$. 

\item[(ii)] For any $s\in (s_\gamma,p^*]$, let
$$K(s)\eqdef \begin{cases}
\frac{(N-p)(s(N-p+\gamma)-Np)}{p(Np-N\gamma-p)}, &\text{if } \gamma\in(0, p-1],\\
s/p(N-p+\gamma)-N, & \text{if } \gamma\in(p-1, \infty),\\
\end{cases}$$

Then, for all $k<K(s)$ (or $k\leq K(s)$ if $\gamma>p-1$), there exists $C>0$ such that for any $u\in \mathcal{D}^{1,p}_\rho(\mathds{R}^N_+)$:
\begin{equation}\label{1}
\int_{\mathds{R}^N_+} (1+x_N)^k|u|^s \mathrm{d}x\leq C \|u\|^s.
\end{equation}
\end{itemize}
\end{theorem}

\begin{proof}
In order to provide an upper bound of the best embedding constant (see Remark \ref{est-best const}) we detail the different constants in the inequalities. We first consider the case $\gamma \in (p-1,\infty)$. Let $u\in \mathcal{D}^{1,p}_\rho(\mathds{R}^N_+)$ and $v=(1+x_N)^\frac{\gamma}{p}u$. Applying \eqref{Hardy p}, we obtain:
$$\int_{\mathds{R}^N _+} |\nabla v|^p\mathrm{d}x \leq 2^{p-1} \int_{\mathds{R}^N _+} [(1+x_N)^\gamma|\nabla u|^p +\Big(\frac{\gamma}{p}\Big)^p(1+x_N)^{\gamma-p}|u|^p]\mathrm{d}x\leq 2^{p-1}\Big[1+\Big(\frac{\gamma}{\gamma-p+1}\Big)^p\Big] \|u\|^p.$$
Consequently, by Sobolev embedding, we get:
\begin{equation}\label{ineq1}
\int_{\mathds{R}^N _+} (1+x_N)^{\frac{\gamma}{p}p^*}|u|^{p^*}\mathrm{d}x\leq S(N,p)^\frac{-p^*}{p} 2^{\frac{p-1}{p}p^*}\Big[1+\Big(\frac{\gamma}{\gamma-p+1}\Big)^p\Big]^\frac{p^*}{p}\|u\|^{p^*}.
\end{equation}
where $S(N,p)$ is the constant of \eqref{GNS}.
We set $s\in (p,p^*]$, $0\leq \alpha =p\frac{p^*-s}{p^*-p}< p$ and $k=\gamma s/p-\alpha=s/p(N-p+\gamma)-N\in(\gamma-p,\gamma p^*/p]$. By Hölder's inequality, we have:
\begin{equation*}
\begin{split}
\int_{\mathds{R}^N _+} (1+x_N)^k|u|^s \mathrm{d}x&=
\int_{\mathds{R}^N _+} (1+x_N)^{\alpha\frac{\gamma-p}{p}} |u|^\alpha (1+x_N)^{k-\alpha\frac{\gamma-p}{p}} |u|^{s-\alpha} \mathrm{d}x\\
&\leq \Big( \int_{\mathds{R}^N _+} (1+x_N)^{\gamma-p}|u|^{p} \mathrm{d}x\Big )^\frac{\alpha}{p} \Big (  \int_{\mathds{R}^N _+} (1+x_N)^{\frac{p}{p-\alpha }(k-\alpha\frac{\gamma-p}{p})}|u|^{\frac{s-\alpha}{p-\alpha}p}\mathrm{d}x \Big )^\frac{p-\alpha}{p}.
\end{split}
\end{equation*}
With the choice of $ \alpha$ and $k$ we have 
$${\frac{s-\alpha}{p-\alpha}p}=p^*,\quad {\frac{p}{p-\alpha }(k-\alpha\frac{\gamma-p}{p})}=\gamma p^*/p.$$
Thus,
\begin{equation}\label{embedfinal}
\int_{\mathds{R}^N _+} (1+x_N)^k|u|^s\mathrm{d}x \leq C \|u\|^s
\end{equation}
and we observe that $k\geq 0 \iff s\geq \frac{Np}{N-p+\gamma}=s_\gamma$. Therefore, we get the desired embedding. Moreover, we obtain $C$ as stated in Remark \ref{est-best const} for $s=s_\gamma$. 

Let us next deal with the case $\gamma \leq p-1$. Taking $u\in  C^\infty_\delta(\mathds{R}^N_+)$, $\alpha\in(1,\frac{p}{p-\gamma})$ and $q\in (1,\infty)$, we have $\alpha<p$ and using Hölder's inequality:
\begin{equation}\label{3b}
\int_{\mathds{R}^N_+}(1+x_N)^{\gamma \alpha/p} |\nabla |u|^q|^\alpha\mathrm{d}x \leq q^\alpha \|(1+x_N)^{\gamma/p} \nabla u\|_{L^{p}(\mathds{R}^N_+)}^\alpha\|u \|_{L^{\alpha(q-1)(p/\alpha)'}(\mathds{R}^N_+)}^{(q-1)\alpha}.
\end{equation}
With $\alpha\in(1,\frac{p}{p-\gamma})$ we have $\gamma\frac{\alpha}{p}>\alpha-1$ and using \eqref{embedfinal}, we obtain:
\begin{equation}\label{4b}
 \|u \|_{L^{qs}(\mathds{R}^N_+)}^{q\alpha}=\Big( \int_{\mathds{R}^N_+} |u|^{qs}\mathrm{d}x \Big)^{\frac{\alpha}{s}}\leq C\int_{\mathds{R}^N_+} (1+x_N)^{\gamma \alpha/p}|\nabla |u|^q|^\alpha\mathrm{d}x
\end{equation}
for $s=\frac{\alpha N}{N-\alpha+\frac{\gamma\alpha}{p}}$. We have $\alpha p -s(p-\alpha)>0$. Consequently,  we can take $q=\frac{\alpha p}{\alpha p -s(p-\alpha)}\in(1,\infty)$ and we have $\alpha(q-1)(p/\alpha)'=q s$. Combining \eqref{3b} and \eqref{4b} we have
\begin{equation}\label{embedf2}
\|u \|_{L^{qs}(\mathds{R}^N_+)}^{q \alpha }\leq C \int_{\mathds{R}^N_+} (1+x_N)^{\gamma \alpha/p}|\nabla |u|^q|^\alpha\mathrm{d}x \leq C q^\alpha \|(1+x_N)^{\gamma/p}\nabla u\|_{L^{p}(\mathds{R}^N_+)}^{\alpha}\|u \|_{L^{qs}(\mathds{R}^N_+)}^{(q-1)\alpha}
\end{equation}
from which we get by density $\mathcal{D}^{1,p}_\rho(\mathds{R}^N_+)\hookrightarrow L^{qs}(\mathds{R}^N_+)$ where
$$qs=\frac{N\alpha p}{N\alpha+p(\gamma\alpha/p-\alpha)}=\frac{Np}{N-p+\gamma}=s_\gamma.$$

Finally, let us prove (ii) for $\gamma\leq p-1$. First, by taking $q=\frac{\alpha p}{\alpha p -\alpha^*(p-\alpha)}$ and $s=\alpha^*$ in \eqref{embedfinal}, we get:
$$\int_{\mathds{R}^N_+} (1+x_N)^{\gamma \alpha^*/p}|u|^{q \alpha^*} \mathrm{d}x \leq C\|u\|^{q\alpha^*}.$$
Observing that $q\alpha^*=p^*$ and letting $\alpha\to (\frac{p}{p-\gamma})^-$ we have
$$\int_{\mathds{R}^N_+} (1+x_N)^{K}|u|^{p^*} \mathrm{d}x \leq C\|u\|^{p^*},$$
for all $K<\gamma ( \frac{p}{p-\gamma})^*/p$. We conclude using \eqref{embedfinal} and Hölder's inequality that:

$$\int_{\mathds{R}^N_+} (1+x_N)^k |u|^s \mathrm{d}x \leq C\|u\|^s,$$
for all $k<K(s)=\frac{\gamma (\frac{p}{p-\gamma})^*}{p}\frac{s-s_\gamma}{p^*-s_\gamma}=\frac{(N-p)(s(N-p+\gamma)-Np)}{p(Np-N\gamma-p)}$ which ends the proof.

\end{proof}

\begin{proposition}
\label{Dinj}
Let $\gamma>0$, $p\leq N$, $\alpha\in(1,p)$, $s_\alpha=\frac{\alpha N}{N-\alpha+\gamma \alpha/p}$, $s\in [s_\alpha,\alpha^*)$ ($s\in [s_\alpha, \alpha^*]$ if $p<N$) and $u\in \mathcal{D}^{1,p}_\rho(\mathds{R}^N_+)$. Then, we have the following inequality
$$\| |u|^q \|_{\mathcal{D}^{1,\alpha}_{\tilde{\rho}}(\mathds{R}^N_+)}^\frac{1}{q}\leq C \|u\| $$
where $\tilde \rho(x_N)=(1+x_N)^\frac{\gamma\alpha}{p}$ and  $q=\frac{\alpha p}{\alpha p -s(p-\alpha)}$. In the case $p=N$, letting $\alpha\to N$ we get $\alpha^*\to \infty $ and letting $s\to \alpha^*$ we get $q\to \infty$.
\end{proposition}
\begin{proof}
Let $u\in C^\infty_\delta(\mathds{R}^N_+)$, $1<\alpha<p\leq N$ and $q\in (1,\infty)$.  
Similarly to \eqref{3b}, we have that:
\begin{equation}\label{3}
\int_{\mathds{R}^N_+} (1+x_N)^{\gamma \alpha/p}|\nabla |u|^q|^\alpha \mathrm{d}x \leq C \|(1+x_N)^{\gamma/p}\nabla u\|_{L^{p}(\mathds{R}^N_+)}^\alpha\|u \|_{L^{\alpha(q-1)(p/\alpha)'}(\mathds{R}^N_+)}^{(q-1)\alpha}.
\end{equation} 
Using Theorem \ref{inj2} with $\alpha<N$ we have:
\begin{equation}\label{4}
\int_{\mathds{R}^N_+} (1+x_N)^{\gamma \alpha/p} |\nabla |u|^q|^\alpha\mathrm{d}x\geq c\Big( \int_{\mathds{R}^N_+} |u|^{qs}\mathrm{d}x \Big)^{\frac{\alpha}{s}}=c\|u \|_{L^{qs}(\mathds{R}^N_+)}^{q\alpha}
\end{equation}
for all $s\in [s_\alpha, \alpha^*)$ ($s\in [s_\alpha, \alpha^*]$ if $p<N$). Taking such a $s$ and with $p\leq N$ we have $\alpha p -s(p-\alpha)>0$. Thus, letting $q=\frac{\alpha p}{\alpha p -s(p-\alpha)}\in(1,\infty)$, we have $\alpha(q-1)(p/\alpha)'=q s$. Combining \eqref{3} and \eqref{4} we get:
$$\|u \|_{L^{qs}(\mathds{R}^N_+)}^{\alpha}\leq C   \|(1+x_N)^{\gamma/p} \nabla u\|_{L^{p}(\mathds{R}^N_+)}^{\alpha}.$$
With \eqref{3}, this yields:
$$ \int_{\mathds{R}^N_+} (1+x_N)^{\gamma \alpha/p}|\nabla |u|^q|^\alpha\mathrm{d}x \leq C \|(1+x_N)^{\gamma/p} \nabla u\|_{L^{p}(\mathds{R}^N_+)}^{q\alpha}.$$
\end{proof}
We now establish new trace embeddings when $\gamma\leq p-1$ which complements Proposition \ref{embedhardy} and Theorem \ref{inj2}:
\begin{theorem}\label{improvetheo2.2}
Let $p<N$, $0<\gamma\leq p-1$. Then $$\mathcal{D}^{1,p}_\rho(\mathds{R}^N_+)\hookrightarrow  L^r(\mathds{R}^{N-1})$$ 
for all $r\in (q_\gamma, p_*]$.
\end{theorem}
\begin{proof}
Setting $1<\alpha<\frac{p}{p-\gamma}\leq p$ and $q=\frac{\alpha p}{\alpha p -s(p-\alpha)}$ with $s\in [s_\alpha,\alpha^*]$, we have $\gamma \alpha/p>\alpha-1$ and using \eqref{Hardy p}, Proposition \ref{Dinj} we get:
$$ \mathcal{D}^{1,p}_\rho(\mathds{R}^N_+)\hookrightarrow  L^r(\mathds{R}^{N-1})$$
for all $\displaystyle{ r\in \Big[ \frac{\alpha^2 p}{\alpha p -s_{\alpha}(p-\alpha)},p_*\Big] }$, noting that $q\alpha_*=p_*$ when $s=\alpha^*$. Thus, the above embedding holds
for all 
$$r\in (\inf_{\alpha\in(1,p/(p-\gamma))} \frac{\alpha^2 p}{\alpha p -s_{\alpha}(p-\alpha)}, p_*].$$
Then, recalling that $0<\gamma\leq p-1$, we easily show that
$$\inf_{\alpha\in(1,p/(p-\gamma))} \frac{\alpha^2 p}{\alpha p -s_{\alpha}(p-\alpha)}=\inf_{\alpha\in(1,p/(p-\gamma))} \frac{pN+p\alpha(\gamma/p-1)}{N-p+\gamma}
=\frac{p(N-1)}{N-p+\gamma}=q_\gamma.$$
\end{proof}

\begin{remark}\label{est-best const}
Let $p>\gamma>p-1$, taking into account steps of the proof of Theorem \ref{inj2} we get the following upper bound for the optimal constant of the Sobolev embedding:
$$\|u\|_{L^{s_\gamma}(\mathds R ^N_+)}\leq C(N,p,\gamma) \|u\|, \text{  for all }u\in \mathcal{D}_\rho^{1,p}(\mathds R ^N_+),$$
where 

\begin{align*}
C(N,p,\gamma)=&\Big(\frac{p}{\gamma-p+1}\Big)^{\frac{p(p^*-s_\gamma)}{s_\gamma(p^*-p)}}S(N,p)^\frac{-p^*(s_\gamma-p)}{ps_\gamma(p^*-p)}2^{\frac{p^\ast(p-1)(s_\gamma-p)}{s_\gamma p(p^*-p)}}\Big(1+\frac{\gamma^p}{(\gamma-p+1)^p}\Big)^\frac{p^*(s_\gamma-p)}{ps_\gamma(p^*-p)}\\
=& \Big(\frac{p}{\gamma-p+1}\Big)^{\gamma/p}S(N,p)^{-\frac{1}{p}(1-\gamma/p)}\Big[2^{p-1}\Big(1+\frac{\gamma^p}{(\gamma-p+1)^p}\Big)\Big]^{\frac{1}{p}(1-\gamma/p)}
\end{align*}

and where $S(N,p)$ is the 
constant in \eqref{GNS}. Note that when $\gamma\to p$ we recover $C(N,p,\gamma)\to C_{p,\gamma}^{-1} $ given in \eqref{Hardy p}. For the case $\gamma\leq p-1$ we get for all $\alpha\in(1,\frac{p}{p-\gamma})$:
$$\|u\|_{L^{s_\gamma}(\mathds R ^N_+)}\leq C(N,p,\gamma,\alpha) \|u\|, \text{  for all }u\in \mathcal{D}_\rho^{1,p}(\mathds R ^N_+),$$
where
$$C(N,p,\gamma,\alpha)= \frac{  Np-\alpha p+\gamma \alpha }{\alpha (N-p+\gamma)} C(N,\alpha, \gamma \alpha /p).$$
Note that relative constants for embeddings in $L^r(\mathds R ^N_+)$ with $r\in (s_\gamma, p^*)$ can be obtained by interpolation with embeddings corresponding to $r=s_\gamma$ and $r=p^*$. Concerning the trace embedding,  the exponent $q_\gamma$ is not attained in the sense that
$$\|u\|_{L^{q}(\mathds R ^{N-1})}\leq C(q) \|u\|, \text{  for all }u\in \mathcal{D}_\rho^{1,p}(\mathds R ^N_+),$$
with 
$$C(q)=O(1)(q(N-p+\gamma)-pN+p)^{(pN+p-\gamma)/q-(N-p+\gamma)/(\gamma-p)}\to \infty$$
as $q\to q_\gamma^+$.
\end{remark}

From Sobolev and trace embeddings satisfied by $\mathcal{D}^{1,p}_\rho(\mathds{R}^N_+)$, we next prove some continuity result with respect to the $x_N$ variable that we will use to show compact trace embeddings.

\begin{lemma}\label{LemCont}
Let $p<N$, $\gamma>0$, $q\in (q_\gamma, p_*)$ then we have the continuous embedding:
$$\mathcal{D}^{1,p}_\rho (\mathds{R}^N_+)\hookrightarrow C([0,\infty),L^q(\mathds{R}^{N-1})).$$
More precisely, for some $\beta>0$, the trace function restricted to $\mathcal{D}^{1,p}_\rho(\mathds{R}^N_+)$ is locally $\beta$-Lipschitz in $L^q$ with respect to $x_N$ variable.
\end{lemma}

\begin{proof}
Let $u\in C^\infty_\delta(\mathds{R}^N_+)$. Let $q\in (q_\gamma, p_*)$ and $t\in (0,+\infty)$ we have:

\begin{equation}\label{estss1}
\begin{split}
\|u( \cdot,t)\|_{L^q(\mathds{R}^{N-1})}^q-\|u( \cdot,0)\|_{L^q(\mathds{R}^{N-1})}^q&=\int_0^t \frac{d}{d_{x_N}}\|u\|_{L^q(\mathds{R}^{N-1})}^q \mathrm{d}x_N \\
&\leq C \int^t_0 \|\nabla u \|_{L^p(\mathds{R}^{N-1})}\|u \|_{L^{p'(q-1)}(\mathds{R}^{N-1})} ^{q-1}\mathrm{d}x_N
\\
&\leq C \Big(\int_0^t \|\nabla u \|_{L^p(\mathds{R}^{N-1})}^\alpha\mathrm{d}x_N\Big)^\frac{1}{\alpha}\Big(\int_0^t\|u \|_{L^{p'(q-1)}(\mathds{R}^{N-1})} ^{\alpha'(q-1)}\mathrm{d}x_N\Big)^\frac{1}{\alpha'}
\end{split}
\end{equation}
for some $\alpha \in(1,p)$. We deduce by H\"older inequality that for $\beta=\displaystyle\frac{1}{\alpha}\left(\frac{p}{\alpha}\right)'>0$, one has:
\begin{equation}\label{estss2}
\Big(\int_0^t \|\nabla u \|_{L^p(\mathds{R}^{N-1})}^\alpha\mathrm{d}x_N\Big)^\frac{1}{\alpha}\leq C t^\beta \|u\|. 
\end{equation}

Using $p<q<p_*$ we get easily $q<p'(q-1)<p^*$. Therefore using interpolation inequality, there exists $\delta\in(0,1)$ independent of $\alpha$ such that
\begin{equation}\label{estss3}
\int_0^t\|u \|_{L^{p'(q-1)}(\mathds{R}^{N-1})} ^{\alpha'(q-1)}\mathrm{d}x_N\leq \int_0^t \|u \|_{L^{p^*}(\mathds{R}^{N-1})} ^{\delta \alpha'(q-1)} \|u \|_{L^{q}(\mathds{R}^{N-1})} ^{\alpha'(q-1)(1-\delta)}\mathrm{d}x_N.
\end{equation}
Since $\alpha\in (1,p)$ and $\delta p'(q-1)=\displaystyle\frac{\delta}{\frac{\delta}{p^*}+\frac{(1-\delta)}{q}}<p^*$, we can take $\alpha$ such that $\delta \alpha'(q-1)=p^*$. Thus, for such $\alpha$, by Theorem \ref{improvetheo2.2}, we infer:
\begin{equation}\label{estss4}
\begin{split}
\int_0^t \|u \|_{L^{p^*}(\mathds{R}^{N-1})} ^{\delta \alpha'(q-1)} \|u \|_{L^{q}(\mathds{R}^{N-1})} ^{\alpha'(q-1)(1-\delta)}\mathrm{d}x_N&=\int_0^t \|u \|_{L^{p^*}(\mathds{R}^{N-1})} ^{p^*} \|u \|_{L^{q}(\mathds{R}^{N-1})} ^{\alpha'(q-1)(1-\delta)}\mathrm{d}x_N\\ 
&\leq C\|u\|^{\alpha'(q-1)(1-\delta)}\int_0^t\|u\|_{L^{p^*}(\mathds{R}^{N-1})}^{p^*}\mathrm{d}x_N\\
&\leq C\|u\|^{\alpha'(q-1)(1-\delta)+p^*}.
\end{split}
\end{equation}
Then, gathering \eqref{estss1}-\eqref{estss4}, we obtain for $x_N>0$:
\begin{equation}\label{cont}
|\|u( \cdot,x_N+h)\|_{L^q(\mathds{R}^{N-1})}^q-\|u(\cdot,x_N)\|_{L^q(\mathds{R}^{N-1})}^q|\leq C h^\beta \|u\|^q.
\end{equation}

\end{proof}

Next, we are interested to show compactness properties in case of cylindrical symmetry. In this aim, we denote $f^*$ a decreasing rearrangement of $f$, it is well known that for any $s\geq 1$: 

$$\|f\|_{L^s} =\|f^* \|_{L^s} , \text{ and } \|f_1-f_2\|_{L^s}  \geq \|f^*_1-f_2^* \|_{L^s},$$

that correspond to the equimeasurability property and the Hardy Littlewood inequality (see \cite{HardyLittlewood} or \cite{Kawohl} for further details), respectively.

\begin{theorem}\label{compa}
Let $p<N$, $\gamma>0$, $q\in(q_\gamma,p_* )$, $s\in (s_\gamma,p^*)$, then the continuous embeddings
$$\mathcal{R}\hookrightarrow L^q(\mathds{R}^{N-1}), \; \mathcal{R}\hookrightarrow L^s(\mathds{R}^{N}_+).$$
are compacts.
\end{theorem}

\begin{proof}
We start by showing the compact embedding in $L^q(\mathds{R}^{N-1})$. Let $\{u_n\}_{n\in\mathds{N}}$ a bounded sequence of $\mathcal{R}$ such that $u_n\rightharpoonup u$ in $\mathcal{R}$ as $n\to \infty$. We set $v_n=({u_n}_{| (0,1)})^*$ the decreasing rearrangement in $x_N$ variable of the restriction of $u_n$ on $\mathds{R}^{N-1}\times (0,1)$. Thus, $v_n$ is nonincreasing in $x_N$ and radially symmetric in $x'$. By the equimeasurability property, note that we have for all $r<\infty$ and a.e. $x'\in \mathds{R}^{N-1}$:
$$\int_0^1 |u_n|^r(x',t)\mathrm{d}t=\int_0^1 |v_n|^r(x',t)\mathrm{d}t,$$
and by nonexpansive property
$$\int_0^1 |({w_1}_{| (0,1)})^*(x',t)-({w_2}_{| (0,1)})^*(x',t)|^r \mathrm{d}t\leq \int_0^1 |w_1-w_2|^r(x',t)\mathrm{d}t.$$
Taking $r=p$, $h\in \mathds{R}$, $0\neq x'_0\in \mathds{R}^{N-1}$,  $w_1(x',t)=u_n( x' +hx'_0,t)$, $w_2(x',t)=u_n(x',t)$ we get by passing to the limit the difference quotients as $h\to 0$:
$$\int_0^1 |\nabla_{x'} v_n|^p\mathrm{d}t\leq c\int_0^1 |\nabla_{x'} u_n|^p\mathrm{d}t$$
from which we easily get that there exists $C>0$ independent of $n$ such that
\begin{equation}\label{boundedsob}
\int_{\mathds{R}^{N-1}\times (0,1)}|\nabla v_n|^p\mathrm{d}t\leq C.
\end{equation}
For $x_N\leq 1$, set $w(x')=\int_0^{x_N} v_n(x',t)\mathrm{d}t$. Then by using \cite[Corollaire II.1, p. 322]{PLLions}, we get:
\begin{equation}\label{estw0}
|w|\leq C\| \nabla_{x'} w \|_{L^p(\mathds{R}^{N-1})}|x'|^{-(N-1-p)/p}.
\end{equation}
Furthermore, by Jensen inequality, we infer

\begin{equation}\label{estw1}
\begin{split}
\| \nabla_{x'} w \|_{L^p(\mathds{R}^{N-1})}^p=\int_{\mathds{R}^{N-1}}\left| \int_0^{x_N} \nabla_{x'} v_n \mathrm{d}t \right|^p\mathrm{d}x'&\leq x_N^{p-1} \int_{\mathds{R}^{N-1}}\int_0^{x_N} |\nabla_{x'} v_n|^p\mathrm{d}t\mathrm{d}x' \\
&\leq c x_N^{p-1}\int_{\mathds{R}^{N-1}}\int_0^{1} |\nabla_{x'} u_n|^p\mathrm{d}t\mathrm{d}x'.
\end{split}
\end{equation}
Since $v_n$ is nonincreasing together with \eqref{estw0} and \eqref{estw1}, we obtain:

$$ v_n(x',x_N)\leq w(x') x_N^{-1}\leq C \|u_n\| |x'|^{-(N-1-p)/p}x_N^{-1/p} .$$
We also have:

$$\int_{\mathds{R}^{N-1}}\int_0^{1} |v_n|^q\mathrm{d}t\mathrm{d}x' =\int_{\mathds{R}^{N-1}}\int_0^{1} |u_n|^q \mathrm{d}t\mathrm{d}x'\leq C\|u_n\|^q $$
for all $q\in (q_\gamma, p^*]$. Thus, we get from \cite[Theorem A.1, p. 338]{Berestycki-Lions} together with \eqref{boundedsob} that up to a subsequence, there exists $v\in \mathcal{R}$ such that  $v_n\to v$ in $L^q(\mathds{R}^{N-1}\times (0,1))$ as $n\to\infty$, for all $q\in (q_\gamma, p^*)$. Furthermore, we can identify $v=({u}_{|(0,1)})^*$. Indeed, using compact embedding on $B\times (0,1)$, on has

$$\int_B \int_0^1 |v_n -({u}_{|(0,1)})^*|^q \mathrm{d}t\mathrm{d}x'\leq \int_B \int_0^1 |u_n -u|^q \mathrm{d}t\mathrm{d}x'\to 0,$$
for all $B$ compact set of $\mathds{R}^{N-1}$. It is not difficult from the uniqueness of such $v$ that the sequence $\{v_n\}_{n\in\mathds{N}}$ converges to $v$ in $L^q(\mathds{R}^{N-1}\times (0,1))$. Finally, using:

$$\int_{\mathds{R}^{N-1}\times (0,1)} |u_n|^q\mathrm{d}x 	=\int_{\mathds{R}^{N-1}\times (0,1)} |v_n|^q\mathrm{d}x \to \int_{\mathds{R}^{N-1}\times (0,1)} |v|^q \mathrm{d}x = \int_{\mathds{R}^{N-1}\times (0,1)} |u|^q \mathrm{d}x$$
we conclude by uniform convexity of $L^q$-space that $u_n\to u$ in $L^q(\mathds{R}^{N-1}\times(0,1))$ as $n\to\infty$, for all $q\in(q_\gamma, p^*)$. Thus $\|u_n\|_{L^q(\mathds{R}^{N-1})} \to \|u\|_{L^q(\mathds{R}^{N-1})}$ as $n\to\infty$ almost everywhere in $(0,1)$, and using Lemma \ref{LemCont}  we have $u_n\to u$ in $L^q(\mathds{R^{N-1})}$ as $n\to\infty$ for all $q\in(q_\gamma,p_*)$.

We now show $u_n\to u$ in $L^s(\mathds{R}^{N}_+)$ as $n\to\infty$ for $s\in(s_\gamma, p^*)$. We already have $u_n\to u$ in $L^s(\mathds{R}^{N-1}\times(0,R))$ as $n\to\infty$ for all $R>0$, and using \eqref{1} we have:
$$\int_R^\infty \int_{\mathds{R}^{N-1}} |u_n-u|^s\mathrm{d}x^\prime \mathrm{d}t\leq C (1+R)^{-k}\| u_n-u\|.$$
Thus taking $K,n$ big enough in the previous inequality, we conclude that $u_n\to u$ in $L^s(\mathds{R}^{N}_+)$ as $n\to\infty$ for all $r\in(s_\gamma, p^*)$.
\end{proof}
We also get
\begin{proposition}\label{inj N=p}
Let $p=N$, $\gamma>0$, $q> q_\gamma$ and $s> s_\gamma$. Then, we have that:

\begin{itemize}
    \item [(i)] the embeddings $\mathcal{D}^{1,p}_\rho(\mathds{R}^N_+)\hookrightarrow  L^{s}(\mathds{R}^N_+), L^{q}(\mathds{R}^{N-1})$ are continuous,
    
    \item [(ii)] the embeddings $\mathcal{R}_\rho(\mathds{R}^N_+)\hookrightarrow  L^{s}(\mathds{R}^N_+), L^{q}(\mathds{R}^{N-1})$ are compact,
    
    \item [(iii)] the embedding $\mathcal{D}^{1,p}_\rho (\mathds{R}^N_+)\hookrightarrow C([0,\infty),L^q(\mathds{R}^{N-1}))$, is compact and more precisely, for some $\beta>0$, the trace operator restricted to $\mathcal{D}^{1,p}_\rho(\mathds{R}^N_+)$ is locally $\beta$-Lipschitz in $L^q$ with respect to $x_N$ variable. 
\end{itemize}
\end{proposition}

\begin{proof}
We only prove the Sobolev embedding, the proof of trace embedding is similar. Using Proposition \ref{Dinj}, we have for $\alpha\in(1,N)$, $s_\alpha\leq s< \alpha^*$, $|u|^q\in \mathcal{D}^{1,\alpha}_{\tilde{\rho}}(\mathds{R}^N_+)$ and 
$$\mathcal{D}^{1,p}_\rho(\mathds{R}^N_+)
\hookrightarrow L^{qs}(\mathds{R}^N_+),$$
where $\tilde \rho(x_N)=(1+x_N)^\frac{\gamma\alpha}{p}$ and $q=\frac{\alpha N}{\alpha N -s(N-\alpha)}$. In particular when $s\to \alpha^*$ we get $sq\to N^*=\infty$. For the case $\gamma\geq p$ the weighted Hardy inequality \eqref{Hardy p} still holds and we get the embedding
$$\mathcal{D}^{1,p}_\rho(\mathds{R}^N_+)\hookrightarrow  L^s(\mathds{R}^N_+),$$
for all $s\geq p$. When $\gamma<p$, using Theorem \ref{inj2} and Proposition \ref{Dinj}, we have the embedding  for all $s>\inf_{\alpha\in(1,N)} \frac{\alpha N s_{\alpha}}{\alpha N -s_{\alpha}(N-\alpha)}$  where $s_\alpha=\frac{\alpha N}{N-\alpha+\gamma \alpha/N}$ and as in the proof of Theorem \ref{inj2} we have:
$$\frac{\alpha N s_{\alpha}}{\alpha N -s_{\alpha}(N-\alpha)}=\frac{N^2}{\gamma}=s_\gamma.$$

Finally arguing as in the proof of Lemma \ref{LemCont}, we get $(iii)$.

\end{proof}

Then, we are ready to give:
\begin{proof}[Proof of Theorem \ref{theo1}]
If $1<p<N$, the conclusion follows directly from Proposition \ref{embedhardy}, Theorem \ref{inj2}, and Theorem \ref{improvetheo2.2}. In the limiting case $p=N$, Proposition \ref{Dinj} yields
\begin{equation*}
\left(\int_{\mathds{R}^N_+}(1+x_N)^{\frac{\gamma \alpha}{p}} |\nabla |u|^q|^\alpha , \mathrm{d}x\right)^{\frac{1}{q\alpha}} \leq C \|u\|,
\end{equation*}
for $\alpha<N$, where $q=\frac{\alpha N}{\alpha N - s_\alpha (N-\alpha)}$.

Applying Theorem \ref{inj2} and using that
\begin{equation*}
\frac{\alpha N s_\alpha}{\alpha N - s_\alpha (N-\alpha)} = s_\gamma,
\end{equation*}
we obtain
\begin{equation*}
\left(\int_{\mathds{R}^N_+}(1+x_N)^k |u|^{s_\gamma} , \mathrm{d}x\right)^{\frac{1}{s_\gamma}} \leq C \|u\|.
\end{equation*}

Then the trace estimate for this case, and hence the full result, follow from (i) of Proposition \ref{inj N=p}.

\end{proof}
\begin{proof}[Proof of Theorem \ref{compa1}]
It follows from Theorem \ref{compa} and (ii) of Proposition \ref{inj N=p}.

\end{proof}
\begin{remark}
As the Hardy-Sobolev inequality \eqref{Hardy p} still holds when $p>N$, we have $\mathcal{D}^{1,p}_\rho(\mathds{R}^N_+)\hookrightarrow  L^{q}(\mathds{R}^{N-1}),$ for $q\in [p,\infty]$ when $\gamma>p-1$ and $\mathcal{D}^{1,p}_\rho(\mathds{R}^N_+)\hookrightarrow L^{s}(\mathds{R}^N_+),$ for $s\in [p,\infty]$ when $\gamma\geq p$. Using Theorems \ref{inj2}, \ref{improvetheo2.2} and Proposition \ref{Dinj} and doing $\alpha\to\frac{p}{p-\gamma}^-$ (resp. $\alpha=N$) if $\frac{p}{p-\gamma}\leq N$ (resp. $\frac{p}{p-\gamma}> N$), in case $\gamma\leq p-1$ the embeddings become $\mathcal{D}^{1,p}_\rho(\mathds{R}^N_+)\hookrightarrow L^{s_\gamma}(\mathds{R}^N_+)$ and $\mathcal{D}^{1,p}_\rho(\mathds{R}^N_+)\hookrightarrow L^{r}(\mathds{R}^{N-1}),$ where
$$r>\begin{cases}
q_\gamma & \text{if }\frac{p}{p-\gamma}\leq N, \\
 \frac{N\gamma}{N-p+\gamma}   & \text{if } \frac{p}{p-\gamma}> N.
\end{cases}$$
\end{remark}

Next, we prove Theorem \ref{TruMos}. A famous result of  N. Trudinger states that for $\Omega$ a domain with finite measure in $\mathds{R}^N$, with $N \geq 2$, there exists a constant $\alpha > 0$ such that
$$\frac{1}{|\Omega|} \int_{\Omega} \exp \left( \alpha |u|^{\frac{N}{N-1}} \right) \mathrm{d}x \le c_0$$
for any $u \in W_0^{1,N}(\Omega)$ with $\|\nabla u\|_{L^N(\Omega)} \le 1$. It was also proved by J. Moser that the best constant is
\begin{equation}\label{alphaN}
\alpha_N=N\omega_{N-1}^\frac{1}{N-1}
\end{equation}
where $\omega_{N-1}$ is the area of the unit $N$-ball. We now investigate Trudinger-Moser inequality for the weighted space $\mathcal{D}^{1,N}_\rho(\mathds{R}^N_+) $. We define $D^{N,s}(\mathds{R}^N)$ as the completion of $C^\infty_0(\mathds R^N)$ under the norm $\|\nabla u \|_{L^N(\mathds R^N)}+\|u\|_{L^s(\mathds R^N)}$. 
We recall the following weighted Trudinger Moser inequality in unbounded domain case:
\begin{theorem}(\cite[Theorem 1.1]{lam-lu-zhang})\label{TruMosunb}
Let $0\leq \beta<N$, $0\leq \alpha<\alpha_N$, $s\geq 1$, $q > s\!\left(1 - \frac{\beta}{N}\right)$ or $q \geq s$ if $\beta = 0 $. Then there exists a positive constant $C> 0$ such that for all $u \in D^{N,s}(\mathds{R}^N)$, $\|\nabla u\|_{L^N(\mathds R^N)} \leq 1$,
$$\int_{\mathds{R}^N} \frac{\exp\!\left(\alpha\!\left(1 - \frac{\beta}{N}\right)|u|^{\frac{N}{N-1}}\right)\,|u|^q}{|x|^\beta}\, \mathrm{d}x \leq C \|u\|_{L^s(\mathds R^N)}^{\,s\left(1 - \frac{\beta}{N}\right)}.$$
Moreover, this constant $\alpha_N$ is the best possible in the sense that if $\alpha \geq \alpha_N$, then the constant $C$ cannot be uniform in functions $u$.
\end{theorem}
Theorem \ref{TruMos} extends it in case of the weighted Sobolev space $\mathcal{D}^{1,N}_\rho(\mathds{R}^N_+)$.


\begin{proof}[Proof of Theorem \ref{TruMos}]
Let $u\in\mathcal{D}^{1,N}_\rho(\mathds{R}^N_+)$ such that $\|u\|\leq 1$. Defining $Su(x',x_N)=u(x',|x_N|)$ we have 
\[
\|\nabla Su \|_{L^N(\mathds{R}^N)}^N=2\|\nabla u\|_{L^N(\mathds{R}^N_+)}^N\leq 2\|u\|^N.
\]
Using Theorem \ref{inj N=p}, we have for $s> s_\gamma$:
\[
\|Su\|_{L^s(\mathds R^N)}^{s}=2\|u\|_{L^s(\mathds R^N_+)}^{s}\leq C \|u\|^s\leq C.
\]
Then, $Su \in D^{N,s}(\mathds{R}^N)$ and $\|\nabla (2^{-1/N}Su)\|_{L^N(\mathds{R}^N)}\leq 1$. Using Theorem \ref{TruMosunb}, we have for $0\leq \beta<N$ and $q > s(1 - \frac{\beta}{N})$ (or $q \geq s$ if $\beta = 0 $):
\begin{align*}
   2 \int_{\mathds{R}^N_+} \frac{\exp\!\left(\alpha\!\left(1 - \frac{\beta}{N}\right)|u|^{\frac{N}{N-1}}\right)\,|u|^q}{|x|^\beta}\, \mathrm{d}x=&\int_{\mathds{R}^N} \frac{\exp\!\left(\alpha\!\left(1 - \frac{\beta}{N}\right)|Su|^{\frac{N}{N-1}}\right)\,|Su|^q}{|x|^\beta}\, \mathrm{d}x\\
   =&2^{q/N}\int_{\mathds{R}^N} \frac{\exp\!\left(2^{\frac{1}{N-1}}\alpha\!\left(1 - \frac{\beta}{N}\right)|2^{-1/N} Su|^{\frac{N}{N-1}}\right)\,|2^{-1/N}Su|^q}{|x|^\beta}\, \mathrm{d}x\\
   \leq& C \|Su\|_{L^s(\mathds R^N)}^{\,s\left(1 - \frac{\beta}{N}\right)}
\end{align*}
for all $0\leq \alpha< 2^{-\frac{1}{N-1}}\alpha_N$, providing the desired result.
\end{proof}

\begin{remark}
Note that when $q=N$ (for example when $\gamma\geq p$), there exists $ C_{\alpha,N}>0$ such that
$$\exp\!\left(\alpha\!\left(1 - \frac{\beta}{N}\right)|u|^{\frac{N}{N-1}}\right)|u|^N \geq C_{\alpha,N}\, \Psi_N\!\left(\alpha\!\left(1 - \frac{\beta}{N}\right)|u|^{\frac{N}{N-1}}\right),$$
where $\Psi$
$$\Psi _N(t)=\exp(t)-\sum^{N-2}_{j=0}\frac{t^j}{j!}.$$
It coincides with Young function in \cite[Theorem 1.5]{Abreu-Furtado-Medeiros2} where the authors prove that
$$\sup_{\{u \in E^{1,N,\gamma} : \|u\|_{E^{1,N,\gamma}} \le 1\}} \int_{\mathds{R}^N_+} \frac{\Psi_N(u)}{(1+x_N)^\beta} \, \mathrm{d}x < +\infty, \quad \forall\, 0 < \alpha < \alpha^*,$$
where $E^{1,p,\gamma}$ is the space defined as the closure of $C_0^\infty(\mathds{R}^N)$ with respect to the norm
$$\|u\|_{E^{1,p,\gamma}} := \left( \int_{\mathds{R}^N_+} \left[ (1+x_N)^\gamma |\nabla u|^p + \frac{|u|^p}{(1+x_N)^{p-\gamma}} \right] \mathrm{d}x \right)^{1/p}.$$ Theorem \ref{TruMos} extends it for $\gamma<N-1$.
\end{remark}

\begin{proof}[Proof of Theorem \ref{bestcons}]
We note $Su$ the Schwartz symmetrisation of $u$ with respect to $x'$, that is the nonincreasing rearrangement of $u(\cdot,x_N)$. Using P\'olya–Szeg\"o inequality and \cite[Lemma 2.2]{NB} we have 
$$S_q=\inf_{V, \|u\|_{L^q(\mathds{R} ^{N-1})}=1} \|Su\|=\inf_{\mathcal{R}, \|v\|_{L^q(\mathds{R} ^{N-1})}=1} \|v\|.$$
Let $\{v_n\}_{n\in\mathds{N}}\subset \mathcal{R}$ be a minimizing sequence, that is $\|v_n\|\to S_q$ as $n\to\infty$ and $\|v_n\|_{L^q(\mathds{R} ^{N-1})}=1$. Using Theorem \ref{compa1}, we get $\|v\|_{L^q(\mathds{R} ^{N-1})}=1$ and since $\{v_n\}_{n\in\mathds{N}}$ is bounded in $\mathcal{R}$ we have $\|v\|\leq \displaystyle\liminf_{n\to\infty} \|v_n\|<\infty$ which implies $\|v\|=S_q$. Arguing similarly when $s\in(s_\gamma,p^*)$, we obtain that the constant $S_s$ is attained.
\end{proof}

\section{Regularity Results}\label{secReg}

\begin{lemma}\label{Linfty}
Let condition $(\rho_0)$ be satisfied, and suppose that $u$ is a 
weak solution. Then the following assertions are valid:
\begin{itemize}
    \item[$(i)$] If $u \in E_q(\mathds{R}^{N-1})$, with $a > 0$, $b \leq 0$, $q \in (1,\infty)$ and $\gamma>0$ with $s \in (s_\gamma,p^\ast)$ or $\gamma\geq p$ with $s \in [p,p^\ast)$, then $u \in L^\infty(\mathds{R}^N_+)$;
    
    \item[$(ii)$] If $u \in E_s(\mathds{R}^N_+)$, with $a \leq 0$, $b > 0$, $s \in  (1,\infty)$ and $\gamma>0$ with $q \in (q_\gamma, p_\ast)$ or $\gamma>p-1$ with $q \in [p, p_\ast)$ then $u \in L^\infty(\mathds{R}^{N-1})$.
\end{itemize}
\end{lemma}

\begin{proof}
We only prove the case $p<N$, the case $p=N$ is similar taking some exponent in $ (s,\infty)$ instead of $p^*$. We begin by dealing with case $(i)$. Since $\gamma>0$, from Theorem \ref{inj2}, we have $E_q(\mathds{R}^{N-1})\hookrightarrow \mathcal{D}^{1,p}_\rho(\mathds{R}^N_+)\hookrightarrow L^s(\mathds{R}^N_+)$. 
Therefore, by a density argument, we can take $E_q(\mathds{R}^{N-1})$ as the set of testing functions in the variational formulation for a weak solution (given in definition \ref{weak solution}). Now, for $m\in \mathds{N}$, let $u_m=\min\{|u|,m\}$. Taking $\varphi=u_m^{pk+1}sgn(u)$ as a testing function in \eqref{varFor}, with $k>0$, 
the assumption $ b \leq 0 $, and the fact that $| u_m| \leq |u| $, we get
\begin{equation}\label{1a inequality}
   \int_{\mathds{R}^N_+} \rho(x_N)|\nabla u|^{p-2}\nabla u\nabla \varphi\mathrm{d}x\leq a\int_{\mathds{R}^{N}_+}|u|^{pk+s}\mathrm{d}x.
\end{equation}
Given that, from \eqref{GNS} and taking into account that $C\|v\|_{L^{p^\ast}(\mathds{R}^N_+)}^p\leq \|\nabla v\|_{L^p(\mathds{R}^N_+)}^p\leq \|v\|^p$, for all $v\in \mathcal{D}^{1,p}_\rho(\mathds{R}^N_+)$, we obtain
\begin{align*}
    \int_{\mathds{R}^N_+} \rho(x_N)|\nabla u|^{p-2} \nabla u\nabla \phi\mathrm{d}x=& \frac{pk+1}{(k+1)^p}\int_{\mathds{R}^{N_+}}\rho(x_N)|\nabla (u_m^{k+1}sgn(u))|^p\mathrm{d}x\\
    \geq& \,C \frac{pk+1}{(k+1)^p}\left(\int_{\mathds{R}^{N}_+}u_m^{(k+1)p^\ast}\mathrm{d}x\right)^{p/p^\ast},
\end{align*}
with $C$ depending on $N$ and $p$. Consequently, applying the inequality above in \eqref{1a inequality}, we have
\begin{equation*}
    C \frac{pk+1}{(k+1)^p}\left(\int_{\mathds{R}^{N}_+}u_m^{(k+1) p^\ast}\mathrm{d}x \right)^{p/p^\ast}\leq a\int_{\mathds{R}^N_+}|u|^{pk+s}\mathrm{d}x =a\int_{\mathds{R}^N_+}|u|^{p(k+1)}|u|^{s-p}\mathrm{d}x .
\end{equation*}
Let $\zeta \in [p,p^\ast)$ such that
$$
\frac{p}{\zeta}+\frac{s-p}{p^*}=1.
$$
Applying the Hölder's inequality, we obtain 
\begin{equation*}
   \left(\int_{\mathds{R}^N_+}u_m^{(k+1) p^\ast}\mathrm{d}x \right)^{1/ p^\ast}\leq C_1\frac{k+1}{(pk+1)^{1/p}} \left(\int_{\mathds{R}^N_+}|u|^{(k+1)\zeta}\mathrm{d}x \right)^{1/\zeta},
\end{equation*}
that is, denoting for simplicity $\|\cdot\|_{L^r(\mathds{R}^N_+)}=\|\cdot\|_r$, 
\begin{equation}\label{recursive}
    \|u_m\|_{(k+1) p^\ast}\leq C_1^{1/(k+1)}\left[\frac{k+1}{(pk+1)^{1/p}}\right]^{1/(k+1)}\|u\|_{(k+1)\zeta},\quad \forall\, k>0.
\end{equation}
Next, we choose $k=k_1>0$ such that $(k_1+1)\zeta=p^\ast$ and define
$$
(k_{n}+1)\zeta=(k_{n-1}+1)p^\ast, \quad n\geq 2.
$$
Note that
\begin{equation}\label{kn}
k_n=\left(\frac{p^\ast}{\zeta}\right)^n-1 \quad \text{and}\quad k_n \longrightarrow \infty, 
\end{equation} 
since $\zeta \in [p, p^\ast)$.
Now, observe that for $k_1$ we have
\begin{equation*}
     \|u_m\|_{(k_1+1) p^\ast}\leq C_1^{1/(k_1+1)}\left[\frac{k_1+1}{(pk_1+1)^{1/p}}\right]^{1/(k_1+1)}\|u\|_{p^\ast}<\infty.
\end{equation*}
Then, by Fatou's Lemma, we get
\begin{equation*}
     \|u\|_{(k_1+1) p^\ast}\leq C_1^{1/(k_1+1)}\left[\frac{k_1+1}{(pk_1+1)^{1/p}}\right]^{1/(k_1+1)}\|u\|_{p^\ast},
\end{equation*}
and by induction we have
\begin{equation*}
     \|u\|_{(k_n+1) p^\ast}\leq C_1^{\sum_{i=1}^n1/(k_i+1)}\prod_{i=1}^n\left[\frac{k_i+1}{(pk_i+1)^{1/p}}\right]^{1/(k_i+1)}\|u\|_{ p^\ast}.
\end{equation*}
Writing
\begin{equation*}
    \left[\frac{k_i+1}{(pk_i+1)^{1/p}}\right]^{1/(k_i+1)}=\left[\left(\frac{k_i+1}{(pk_i+1)^{1/p}}\right)^{1/\sqrt{k_i+1}}\right]^{1/\sqrt{k_i+1}},
\end{equation*}
and given that $\left(\frac{k+1}{(pk+1)^{1/p}}\right)^{1/\sqrt{k+1}}>1$ for $k>0$ and $\displaystyle\lim_{k\longrightarrow \infty}\left(\frac{k+1}{(pk+1)^{1/p}}\right)^{1/\sqrt{k+1}}=1$, there exists $C_2>1$ such that
\begin{equation*}
     \|u\|_{(k_n+1) p^\ast}\leq C_1^{\sum_{i=1}^n\frac{1} {k_i+1}}C_2^{\sum_{i=1}^n\frac{1}{\sqrt{k_i+1}}}\|u\|_{ p^\ast}.
\end{equation*}
Now, observe that by \eqref{kn} and the fact that $\zeta<p^\ast$, letting $n\longrightarrow \infty$, we obtain $\|u\|_{\infty}\leq C \|u\|_{ p^\ast}$ and hence $u \in L^\infty(\mathds{R}^N_+)$. 

For item $(ii)$, we can apply the same reasoning and obtain
\begin{equation*}
   \left(\int_{\mathds{R}^{N-1}}u_m^{(k+1) p_\ast}\mathrm{d}x^\prime \right)^{1/ p_\ast}\leq C\frac{k+1}{(pk+1)^{1/p}} \left(\int_{\mathds{R}^N_+}|u|^{(k+1)\eta}\mathrm{d}x \right)^{1/\eta},
\end{equation*}
with $\eta \in [p,p_\ast)$ such that $p/\eta+
(q-p)/p_\ast=1.$ Defining $k_n=\left(p_\ast/\eta\right)^n-1$ we can reproduce the interaction argument and obtain $\|u\|_{L^\infty(\mathds{R}^{N-1})}\leq C \|u\|_{L^{p_\ast}(\mathds{R}^{N-1})}$. The case $p = N$ follows by a similar argument, replacing $p_\ast$ with any exponent in $(q, \infty)$. This completes the proof of item $(ii)$ and hence the result.
\end{proof}

\begin{lemma}\label{Linfty 2}
Assume condition $(\rho_0)$ holds with $\gamma\geq 0$. Then, if $u$ is a 
weak solution for \eqref{PG}, in each of the following cases:
\begin{itemize}
\item[$(i)$] If $u \in E_q(\mathds{R}^{N-1})$, with $a > 0$, $b \leq 0$, $s=p^\ast$ and $q\in(1,\infty)$, then $u \in L^\infty(\mathds{R}^N_+)$;
\item[$(ii)$] If $u \in E_s(\mathds{R}^N_+)$, with $a \leq 0$, $b > 0$, $s\in(1,\infty)$ and $q=p_\ast$, then $u \in L^\infty(\mathds{R}^{N-1})$.
\end{itemize}
\end{lemma}

\begin{proof} For $\gamma \geq 0$, observe that, by \eqref{GNS}, \eqref{Trace-inequality}, and condition $(\rho_0)$, we have the continuous embeddings
\begin{equation*}
E_q(\mathds{R}^{N-1}) \hookrightarrow L^{p^\ast}(\mathds{R}^N_+)
\quad \text{and} \quad
E_s(\mathds{R}^N_+) \hookrightarrow L^{p_\ast}(\mathds{R}^{N-1}).
\end{equation*}
Consequently, in each case, the weak formulation can be extended to testing functions in
$E_q(\mathds{R}^{N-1})$ and $E_s(\mathds{R}^N_+)$, respectively.
Let $u$ be a 
weak solution for \eqref{PG}, $d>0$, $k>1$ and let us denote $\varphi_{k,d} = \min \{ |u|^{pd+1}, k^p |u|\}sgn(u)$ and $\psi_{k,d} = \min \{ |u|^{d+1}, k |u|\}sgn(u)$. We infer that
\begin{equation*}
    \nabla\varphi_{k,d}= \left\lbrace
		\begin{aligned}
			&(pd+1)|u|^{pd}\nabla u\quad&\mbox{for }&\,|u|^d\leq k,\\
		&k^p\nabla u,\quad &\mbox{for}&\,|u|^d\geq k
		\end{aligned}
		\right.
        \quad \text{and}\quad 
   \nabla \psi_{k,d}=\left\lbrace
		\begin{aligned}
			&(d+1)|u|^d \nabla u\quad&\mbox{for }&\,|u|^d\leq k,\\
		&k\nabla u,\quad &\mbox{for}&\,|u|^d\geq k.
		\end{aligned}
		\right.
\end{equation*}
Now, observe that
\begin{align*}
    \int_{\mathds{R}^N_+}\rho(x_N)|\nabla u|^{p-2}\nabla u \nabla \varphi_{k,d} \mathrm{d}x=& (pd+1)\int_{\{|u|^d\leq k\}}\rho(x_N) |u|^{pd}|\nabla u|^p\mathrm{d}x+k^p\int_{\{|u|^d\geq k\}}\rho(x_N)|\nabla u|^p\mathrm{d}x\\ 
    =&\frac{pd+1}{(d+1)^p}\int_{\{|u|^d\leq k\}}\rho(x_N)(d+1)^p |u|^{pd}|\nabla u|^p\mathrm{d}x+k^p\int_{\{|u|^d\geq k\}}\rho(x_N)|\nabla u|^p\mathrm{d}x\\
    \geq& \frac{pd+1}{(d+1)^p}\int_{\mathds{R}^N_+}\rho(x_N)|\nabla \psi_{k,d}|^p \mathrm{d}x,
\end{align*}
where we have used that $pd+1\leq (d+1)^p$. Then, testing the weak formulation \eqref{varFor} with $\varphi_{k,d}$, we obtain
\begin{equation}\label{psi grad ineq}
    \int_{\mathds{R}^N_+}\rho(x_N)|\nabla \psi_{k,d}|^p\mathrm{d}x\leq \frac{(d+1)^p}{dp+1}\left(a\int_{\mathds{R}^N_+} |u|^{s-2}u|\varphi_{k,d}| \mathrm{d}x+b\int_{\mathds{R}^{N-1}}|u|^{q-2}u|\varphi_{k,d}| \mathrm{d}x^\prime\right).
\end{equation}

Consider now the item $(i)$. Since $a>0\geq b$, $\rho\geq 1$ by $\rho_0$ and \eqref{GNS} holds for $\psi_{k,d}$, we obtain
\begin{align}\label{key inequality 1}
\nonumber \left(\int_{\mathds{R}^N_+}|\psi_{k,d}|^{p^\ast}\mathrm{d}x\right)^{p/p^*}\leq& C_0\frac{(d+1)^p}{dp+1}\int_{\mathds{R}^N_+}|u|^{p^\ast-2}u|\varphi_{k,d} |\mathrm{d}x\\
=& C_0\frac{(d+1)^p}{dp+1}\int_{\mathds{R}^N_+}|\psi_{k,d}|^p |u|^{p^\ast-p} \mathrm{d}x,
\end{align}
with $C_0>0$ depending on $N,p$ and $a>0$.

\textbf{Claim:} Let $d_1>0$ such that $p(d_1+1)=p^\ast$, then $u \in L^{(d_1+1)p^\ast}(\mathds{R}^N_+)$.

In fact, let $R>0$ to be chosen later. By Hölder's inequality and \eqref{key inequality 1} for $d=d_1$, we have
\begin{align*}
    \int_{\mathds{R}^N_+}|\psi_{k,d}|^p |u|^{p^\ast-p}\mathrm{d}x=& \int_{\{|u|<R\}}|\psi_{k,d}|^p |u|^{p^\ast-p}\mathrm{d}x+ \int_{\{|u|>R\}}|\psi_{k,d}|^p |u|^{p^\ast-p}\mathrm{d}x\\
    \leq& R^{p^\ast-p}\int_{\mathds{R}^N_+}|\psi_{k,d}|^p \mathrm{d}x+\left(\int_{\mathds{R}^N_+} |\psi_{k,d}|^{p^\ast}\mathrm{d}x\right)^{p/p^\ast}\left(\int_{\{|u|>R\}} |u|^{p^\ast}\mathrm{d}x\right)^{(p^\ast-p)/p^\ast}.
\end{align*}
Now, choose $R>0$ such that $\|u\|_{p^\ast,\{|u|>R\}}^{p^\ast-p}\leq (d_1p+1)/2C_0(d_1+1)^p$, then, by the inequality above and \eqref{key inequality 1}, we get
\begin{align*}
\left(\int_{\mathds{R}^N_+}|\psi_{k,d}|^{p^\ast}\mathrm{d}x\right)^{p/p^*}\leq C_0\frac{(d_1+1)^p}{d_1p+1}\left[ R^{p^\ast-p}\int_{\mathds{R}^N_+}|\psi_{k,d}|^p \mathrm{d}x+\frac{d_1p+1}{2C_0(d_1+1)^p}\left(\int_{\mathds{R}^N_+} |\psi_{k,d}|^{p^\ast}\mathrm{d}x\right)^{p/p^\ast}\right].
\end{align*}
Thus, we have
\begin{align*}
\left(\int_{\mathds{R}^N_+}|\psi_{k,d}|^{p^\ast}\mathrm{d}x\right)^{p/p^*}\leq& C_0\frac{(d_1+1)^p}{d_1p+1} R^{p^\ast-p}\int_{\mathds{R}^N_+}|\psi_{k,d}|^p \mathrm{d}x\\
\leq& C_0\frac{(d_1+1)^p}{d_1p+1} R^{p^\ast-p}\int_{\mathds{R}^N_+}|u|^{(d_1+1)p} \mathrm{d}x,
\end{align*}
with the right hand side finite, given that $(d_1+1)p=p^\ast$. Since $\psi_{k,d} \longrightarrow |u|^{d}u$ $a.e.$ in $\mathds{R}^N_+$ for all $d>0$ when $k \longrightarrow \infty$, by Fatou's Lemma we obtain
\begin{equation*}
\left(\int_{\mathds{R}^N_+}|u|^{(d_1+1)p^\ast}\mathrm{d}x\right)^{p/p^*} \leq C_0\frac{(d_1+1)^p}{d_1p+1} R^{p^\ast-p}\int_{\mathds{R}^N_+}|u|^{p^\ast} \mathrm{d}x<\infty,
\end{equation*}
which concludes the claim.

Now, by \eqref{key inequality 1}, we obtain
\begin{equation}\label{key inequality 2}
\left(\int_{\mathds{R}^N_+}|\psi_{k,d}|^{p^\ast}\mathrm{d}x\right)^{1/p^\ast d}\leq C_1^{1/d}\left[\frac{d+1}{(dp+1)^{1/p}}\right]^{1/d}\left(\int_{\mathds{R}^N_+} |u|^{p(d+1)+p^\ast-p} \mathrm{d}x\right)^{1/pd}, \quad \forall d>0,
\end{equation}
since $|\psi_{k,d}|\leq |u|^{d+1}$. Now, define $d_{n+1}=(p^\ast/p)^n d_1$, for $n \geq 1$ and observe that
\begin{equation}\label{dn}
    p(d_{n+1}+1)+p^\ast-p=(d_n+1)p^\ast\quad \text{and}\quad pd_{n+1}=p^\ast d_n.
\end{equation}
For $n=1$, by \eqref{key inequality 2} we have
\begin{align*}
\left(\int_{\mathds{R}^N_+}|\psi_{k,d_2}|^{p^\ast}\mathrm{d}x\right)^{1/p^\ast d_2}\leq&  C_1^{1/d_2}\left[\frac{d_2+1}{(d_2 p+1)^{1/p}}\right]^{1/d_2}\left(\int_{\mathds{R}^N_+} |u|^{p(d_2+1)+p^\ast-p} \mathrm{d}x\right)^{1/pd_2}\\
=& C_1^{1/d_2}\left[\frac{d_2+1}{(d_2 p+1)^{1/p}}\right]^{1/d_2}\left(\int_{\mathds{R}^N_+} |u|^{(d_1+1)p^\ast} \mathrm{d}x\right)^{1/p^\ast d_1}.
\end{align*}
Thus, by Fatou's Lemma, we have
\begin{equation*}
\left(\int_{\mathds{R}^N_+}|u|^{(d_2+1)p^\ast}\mathrm{d}x\right)^{1/p^\ast d_2}\leq C_1^{1/d_2}\left[\frac{d_2+1}{(d_2 p+1)^{1/p}}\right]^{1/d_2}\left(\int_{\mathds{R}^N_+} |u|^{(d_1+1)p^\ast} \mathrm{d}x\right)^{1/p^\ast d_1},
\end{equation*}
and using \eqref{key inequality 2} and \eqref{dn}, by induction, we obtain
\begin{equation*}
\left(\int_{\mathds{R}^N_+}|u|^{(d_{n+1}+1)p^\ast}\mathrm{d}x\right)^{1/p^\ast d_{n+1}}\leq C_1^{\sum_{i=2}^{n+1}{d_i^{-1}}}\prod_{i=2}^{n+1}\left[\frac{d_i+1}{(d_i p+1)^{1/p}}\right]^{1/d_i}\left(\int_{\mathds{R}^N_+} |u|^{(d_1+1)p^\ast} \mathrm{d}x\right)^{1/p^\ast d_1},\quad n\geq 1.
\end{equation*}
Arguing similarly to Lemma \ref{Linfty}, letting $n\longrightarrow \infty$, we have $u \in L^\infty(\mathds{R}^N_+)$ with
\begin{equation*}
    \|u\|_{L^\infty(\mathds{R}^N_+)}\leq C \|u\|_{L^{(d_1+1)p^\ast}(\mathds{R}^N_+)}^{(d_1+1)/d_1}.
\end{equation*}
Now, consider item (ii). Similarly to the previous case, from \eqref{psi grad ineq} and \eqref{Trace-inequality}, we have
\begin{equation*} \left(\int_{\mathds{R}^{N-1}}|\psi_{k,d}|^{p_\ast}\mathrm{d}x\right)^{p/p_\ast}\leq K_0\frac{(d+1)^p}{dp+1}\int_{\mathds{R}^{N-1}}|\psi_{k,d}|^p |u|^{p_\ast-p} \mathrm{d}x,
\end{equation*}
with $K_0>0$ depending on $N,p$ and $b>0$. Thus, following the same reasoning of the first case with 
\begin{equation*}
    (d_1+1)p=p_\ast\quad \text{and} \quad d_{n+1}=\left(\frac{p_\ast}{p}\right)^n d_1,\quad \text{for $n\geq 1$},
\end{equation*}
we obtain $u \in L^\infty(\mathds{R}^{N-1})$, with 
\begin{equation*}
    \|u\|_{L^\infty(\mathds{R}^{N-1})}\leq C \|u\|_{L^{(d_1+1)p_\ast}(\mathds{R}^{N-1})}^{(d_1+1)/d_1}<\infty.
\end{equation*}
\end{proof}
We are now ready to prove the Hölder regularity result.

\begin{proof}[Proof of Theorem \ref{C1alpha regularity}] Suppose that item $(i)$ holds. By Lemma \ref{Linfty} we have $u \in L^\infty(\mathds{R}^N_+)$. Observe that the equation \eqref{PG} with $\rho \in C^{0,\beta}_{\mathrm{loc}}[0,\infty)$ satisfies the structural conditions of regularity results in \cite{Lieberman}. Then, the conclusion follows by \cite[Theorem 2]{Lieberman}. Conclusion in case of item $(ii)$ follows similarly by applying Lemma \ref{Linfty 2}.

For item $(iii)$, by Lemma \ref{Linfty} we obtain $u \in L^\infty(\mathds{R}^{N-1})$. Let $k$ be such that $\|u\|_{L^\infty(\mathds{R}^{N-1})}<k$ and consider $T_ku=(u-k)_+$ as a testing function, we obtain
\begin{equation*}
    0\leq \int_{\Omega_k}\rho(x_N)|\nabla u|^p\mathrm{d}x=a\int_{\mathds{R}^N_+}u^{p-1}T_ku\mathrm{d}x+b\int_{\mathds{R}^{N-1}}u^{q-1}T_ku\mathrm{d}x^\prime\leq 0,
\end{equation*}
where $\Omega_k=\{x\in \mathds{R}^N_+, u(x)>k\}$. Hence, either $u$ is constant or $u\leq k$ a.e. in $\mathds{R}^N_+$. Testing instead with $(-u - k)_+$, we deduce that either $u$ is constant or $u \geq -k$. Combining this with the previous case, we conclude that $-k\leq u\leq k$, which implies $u \in L^\infty(\mathds{R}^N_+)$. Therefore, by \cite[Theorem 2]{Lieberman}, we get the validity of conclusion in case of item $(iii)$ and with the same reasoning we prove it in case of item $(iv)$, concluding the result. Using Green's formula (see \cite{Ilyasov-Takac}) we have that $u$ verifies the boundary condition pointwisely.
\end{proof}


Now we focus on the proof of the $W^{2,r}_{\mathrm{loc}}$ regularity for weak solutions of \eqref{PG}. First of all, we adopt the following notations: Consider the vector field $F= (F_1,\cdots, F_N ):\mathds{R}^N \longrightarrow \mathds{R}^N$ given by 
\[
F(\eta) = |\eta|^{p-2}\eta,
\]
with $\eta=(\eta_1,\cdots, \eta_N)\in \mathds{R}^N$ and the entries $F_i(\eta)=|\eta|^{p-2}\eta_i$, for $i=1,\cdots,N$. The Jacobian matrix of $F$ is given by 
\begin{equation}\label{Jacobian of F}
F^\prime(\eta) = |\eta|^{p-2} \left( I + (p-2)\frac{\eta \otimes \eta}{|\eta|^2} \right), \quad \text{for } \eta \in \mathds{R}^N \setminus \{0\}
\end{equation}
where $\xi \otimes \eta = (\xi_i \eta_j)_{i,j=1}^N$, $I$ denote the identity matrix $N \times N$, and one has $ \partial F_i/\partial \eta_j \in C^0(\mathds{R}^N \setminus \{0\})$. Furthermore, it is well known that there exists $\lambda, \Lambda>0$ such that

\begin{equation}\label{partial Fi}
\sum_{i,j=1}^N 
\left|\frac{\partial F_i}{\partial \eta_j}(\eta)\right| \leq 
\Lambda |\eta|^{p-2}
\end{equation}
and
\begin{equation}\label{Ellipticity}
\lambda |\eta|^{p-2}|\xi|^2 
\leq
\langle F^\prime(\eta)\xi, \xi \rangle
= 
\sum_{i,j=1}^N \frac{\partial F_i}{\partial \eta_j}(\eta) \, \xi_i \xi_j
\leq
\Lambda |\eta|^{p-2}|\xi|^2,
\end{equation}
for all $ \eta \in \mathds{R}^N \setminus \{0\} $ and all $ \xi \in \mathds{R}^N $, where 
$\langle \cdot, \cdot \rangle$ denotes the standard Euclidean inner product of $\mathds{R}^N$.

Now we introduce a cut-off function inspired by the construction in \cite{Ilyasov-Takac}. Let $\Omega \subset \mathds{R}^N$ an open domain and let
\[
d(x) = \operatorname{dist}(x,\partial\Omega)
\]
denote the distance from a point $x \in \Omega$ to the boundary $\partial\Omega$.  
For $\delta > 0$ sufficiently small, set
\[
\Omega_\delta := \{\, x \in \Omega : d(x) < \delta \,\}.
\] 
Let $\varphi \in C^{1}_{0}(\Omega)$ be a nonnegative function whose support is contained in the region
\[
\Omega_\delta^\prime := \Omega \setminus \Omega_\delta 
   = \{\, x \in \Omega : d(x) > \delta \,\}.
\]  
Then, for every vector $h \in \mathds{R}^N$ with $0 < |h| < \delta$, the difference quotient
\[
D_h (\varphi(x) )= \frac{\varphi(x+h)-\varphi(x)}{|h|}, 
   \qquad x \in \Omega,
\]
belongs to $C^{1}_{0}(\Omega)$. Assuming $\delta > 0$ sufficiently small, define
\[
d_\delta(x) :=
\begin{cases}
0, & x \in \Omega_\delta,\\[4pt]
\operatorname{dist}(x,\Omega_\delta), & x \in \Omega \setminus \Omega_\delta,
\end{cases}
\]
and, for any $\sigma \in (0,\delta)$, consider 
\begin{equation*}
\mathcal{O}^{\delta}_{\sigma}:= \{\, x \in \Omega_\delta^\prime: d_\delta(x) < \sigma \,\}.
\end{equation*}
This set lies in a neighborhood of $\partial\Omega$ and satisfies 
$\overline{\mathcal{O}^{\delta}_{\sigma}} \subset \Omega$.

We define
\[
\varphi^{\delta}_{\sigma}(x) =
\begin{cases}
0, & x \in \overline{\Omega_\delta},\\
\left(\sigma^{-1}d_\delta(x)\right)^{2}, & x \in \mathcal{O}^{\delta}_{\sigma},\\
1, & x \in \Omega \setminus (\overline{\Omega_\delta} \cup \mathcal{O}^{\delta}_{\sigma}),
\end{cases}
\]
which satisfies $0 \le \varphi^{\delta}_{\sigma} \le 1$ in $\Omega$ and  
$\varphi^{\delta}_{\sigma} \in W^{1,r}_{0}(\Omega)$ for every $r > N$.

Moreover,
\[
\nabla \varphi^{\delta}_{\sigma}(x) =
\begin{cases}
2\sigma^{-2}d_\delta(x) \nabla d_\delta(x), 
   & x \in \mathcal{O}^{\delta}_{\sigma},\\
0, & x \in \Omega \setminus \overline{\mathcal{O}^{\delta}_{\sigma}},
\end{cases}
\]
as established in Gilbarg–Trudinger \cite[Theorem~7.8, p.~153]{GT}.  
Consequently, there exists a constant $C>0$, independent of  
$0 < \sigma < \delta$, such that
\begin{equation}\label{grad varphi ineq}
|\nabla \varphi^{\delta}_{\sigma}(x)|^{2}
   \le C\, \sigma^{-2}\, \varphi^{\delta}_{\sigma}(x)
   \qquad \text{for all } x \in \Omega \setminus \partial\mathcal{O}^{\delta}_{\sigma}.
\end{equation}
Now, we can prove the Sobolev regularity given in Theorem \ref{W2regularity}:

\begin{proof}[Proof of Theorem \ref{W2regularity}] Let $1<p\leq 2$. In each of cases $(i)-(iv)$, we apply the regularity results of
\cite[Theorem, p.~840]{DeThelin} 
to obtain $u \in W^{2,p}_{\mathrm{loc}}(\mathds{R}^N_+)$, concluding the item $(1)$.

Now consider $p>2$ and $A:\mathds{R}_+\times \mathds{R}^N\longrightarrow \mathds{R}^N$ given by $A(t,\eta)=\rho(t)F(\eta)=\rho(t)|\eta|^{p-2}\eta$. The Jacobian matrix of $A$ is the $N \times N+1$ matrix given by
\[
A^\prime(t,\eta)=
\bigl[\, \rho'(t)F(\eta) \;\big|\; \rho(t)F'(\eta) \,\bigr],
\]
and for all $(s,\xi) \in \mathds{R}\times \mathds{R}^N$, 
\begin{equation}\label{Derivative of A(t,eta)}
A^\prime(t,\eta)\cdot (s,\xi)= s\rho^\prime(t)F(\eta)+\rho(t)F^\prime(\eta)\cdot \xi \in \mathds{R}^N.
\end{equation}
It is sufficient to prove that $u \in W^{2,2}_{\mathrm{loc}}(\Omega_\delta^\prime\cap U) $ for any bounded open set $\Omega \subset \mathds{R}^N_+$. Let $\varphi \in C^1_0(\mathds{R}^N)$ be such that $ \mathrm{supp}(\varphi) \subset \Omega_\delta^\prime$. Then, for $h=(h_1,\cdots, h_N) \in \mathds{R}^N$ such that $0<|h|<\delta$, we can test the weak formulation \eqref{varFor} with $D_h(\varphi) \in C^1_0(\Omega)$. Consequently,
\begin{equation*}    \int_{\mathds{R}^N_+}A(x_N, \nabla u)\cdot \nabla(D_h(\varphi))\mathrm{d}x = a\int_{\mathds{R}^N_+}|u|^{s-2}uD_h(\varphi) \mathrm{d}x. 
\end{equation*}
After a change of variables and subsequently replacing $h$ by $-h$, we obtain
\begin{equation}\label{1 step}
    \int_{\Omega_\delta^\prime}\frac{A(x_N+h_N, \nabla u(x+h))-A(x_N, \nabla u(x))}{|h|}\cdot \nabla \varphi \mathrm{d}x= a \int_{\Omega_\delta^\prime} D_h(|u|^{s-2}u)\varphi\mathrm{d}x.
\end{equation}
Now, given that $\rho \in W^{1,\infty}_{\mathrm{loc}}(0,\infty)$, we can apply Taylor formula to get
\begin{equation}\label{Taylor formula}
A(t_2,\eta_2)-A(t_1,\eta_1)=\int_0^1 A^\prime((t_1,\eta_1)+ z(t_2-t_1, \eta_2-\eta_1))\cdot (t_2-t_1, \eta_2-\eta_1)\mathrm{d}z \in \mathds{R}^N.    
\end{equation}

Setting $t_1=x_N$, $t_2=x_N+h_N$, $\eta_1= \nabla u(x)$ and $\eta_2 = \nabla u(x+h)$ in \eqref{Taylor formula}, we deduce
\begin{align*}
    \frac{A(x_N+h_N, \nabla u(x+h))-A(x_N, \nabla u(x))}{|h|}=& \int_0^1 A^\prime (x_N+zh_N, (1-z)\nabla u(x)+z\nabla u(x+h))\cdot \left(\frac{h_N}{|h|}, D_h(\nabla u)\right)\mathrm{d}z.
\end{align*}
Thus, we obtain
\begin{equation}\label{Key ineq varphi}
  \int_{\Omega_\delta^\prime}\int_0^1 A^\prime (x_N+zh_N, (1-z)\nabla u(x)+z\nabla u(x+h))\cdot \left(\frac{h_N}{|h|}, D_h(\nabla u)\right)\cdot \nabla \varphi \mathrm{d}z\mathrm{d}x = a\int_{\Omega_\delta^\prime} D_h(|u|^{s-2}u)\varphi \mathrm{d}x,\quad \forall\varphi \in C^1_0(\Omega_\delta^\prime). 
\end{equation}
By Theorem \ref{C1alpha regularity}, $u \in C^{1,\alpha}(\overline{\Omega})$, then we can replace $\varphi$ by $(D_hu) \varphi$ in \eqref{Key ineq varphi} and get
\begin{align*}
   & \int_{\Omega_\delta^\prime}\int_0^1 A^\prime (x_N+zh_N, (1-z)\nabla u(x)+s\nabla u(x+h))\cdot \left(\frac{h_N}{|h|}, D_h(\nabla u)\right)\cdot D_h(\nabla u) \varphi \mathrm{d}z\mathrm{d}x\\
    & +\int_{\Omega_\delta^\prime}\int_0^1 A^\prime (x_N+zh_N, (1-z)\nabla u(x)+z\nabla u(x+h))\cdot \left(\frac{h_N}{|h|}, D_h(\nabla u)\right)\cdot \nabla \varphi D_h(u) \mathrm{d}z\mathrm{d}x \\
    & =a\int_{\Omega_\delta^\prime} D_h(|u|^{s-2}u)D_h(u) \varphi \mathrm{d}x.
\end{align*}
Thus, replacing $\varphi$ by $\varphi_\sigma^\delta$ using \eqref{Derivative of A(t,eta)}, we obtain
\begin{align*}\label{key inequality }
    \int_{\Omega_\delta^\prime}\int_0^1 \rho^\prime(x_N+sh_N)\frac{h_N}{|h|}\langle F((1-z)\nabla u(x)+z\nabla u(x+h)), D_h(\nabla u) \rangle \varphi_\sigma^\delta \mathrm{d}z\mathrm{d}x\\
    +\int_{\Omega_\delta^\prime}\int_0^1 \rho(x_N)\langle F^\prime((1-z)\nabla u(x)+z\nabla u(x+h))D_h(\nabla u),D_h(\nabla u)\rangle \varphi_\sigma^\delta \mathrm{d}z\mathrm{d}x\\
   +\int_{\Omega_\delta^\prime}\int_0^1 \rho^\prime(x_N+sh_N)\frac{h_N}{|h|}\langle F((1-z)\nabla u(x)+z\nabla u(x+h)), \nabla \varphi_\sigma^\delta \rangle D_h(u)\mathrm{d}z\mathrm{d}x\\
    +\int_{\Omega_\delta^\prime}\int_0^1\rho(x_N)\langle F^\prime((1-z)\nabla u(x)+z\nabla u(x+h))\cdot D_h(\nabla u), \nabla \varphi_\sigma^\delta \rangle D_h(u) \mathrm{d}z\mathrm{d}x\\
     = a\int_{\Omega_\delta^\prime} D_h(|u|^{s-2}u)D_h(u) \varphi_\sigma^\delta \mathrm{d}x.
\end{align*}
Rewriting the equation above as
\begin{equation}\label{I1...I4}
    I_1+I_2+I_3+I_4 = a\int_{\Omega_\delta^\prime} D_h(|u|^{s-2}u)D_h(u) \varphi_\sigma^\delta \mathrm{d}x,
\end{equation}
note firstly that from  $(\rho_0)$ and  \eqref{Ellipticity}, we obtain 
\begin{align}\label{I2 ineq}
    \nonumber I_2\geq&  C_0\int_{\Omega_\delta^\prime}\int_0^1 \langle F^\prime((1-z)\nabla u(x)+z\nabla u(x+h))D_h(\nabla u),D_h(\nabla u)\rangle \varphi_\sigma^\delta\mathrm{d}z\mathrm{d}x\\
   \nonumber \geq&C_0\lambda  \int_{\Omega_\delta^\prime}\int_0^1|(1-z)\nabla u(x)+z\nabla u(x+h)|^{p-2}\mathrm{d}z|D_h(\nabla u)|^2 \varphi_\sigma^\delta\mathrm{d}x\\
    =&C_0\lambda \int_{\Omega_\delta^\prime}\tilde{a}(x,h)|D_h(\nabla u)|^2\varphi_\sigma^\delta \mathrm{d}x,
\end{align}
where 
\begin{equation*}
    \tilde{a}(x,h):= \int_0^1|(1-z)\nabla u(x)+z\nabla u(x+h)|^{p-2}\mathrm{d}z.
\end{equation*}
Now, since $\rho \in W^{1,\infty}_{\mathrm{loc}}(0,\infty)$ and $h_N/|h|\leq 1$, for $I_1$ we have
\begin{align*}
    |I_1|\leq& \|\rho^\prime\|_{L^\infty(\Omega)} \int_{\Omega_\delta^\prime}\int_0^1|(1-z)\nabla u(x)+z\nabla u(x+h)|^{p-2}|\langle (1-z)\nabla u(x)+z\nabla u(x+h),D_h(\nabla u)\rangle|\varphi_\sigma^\delta\mathrm{d}z\mathrm{d}x\\
    \leq& \|\rho^\prime\|_{L^\infty(\Omega)}\|\nabla u\|_{L^\infty(\Omega)} \int_{\Omega_\delta^\prime}\tilde{a}(x,h)|D_h(\nabla u)|\varphi_\sigma^\delta \mathrm{d}x.    
\end{align*}
Applying the Young inequality
\begin{equation}\label{Young}
    |cd|\leq \varepsilon c^2+\frac{d^2}{4\varepsilon}, \quad \forall \varepsilon >0.
\end{equation}
with $\varepsilon=\lambda/(4\|\rho^\prime\|_{L^\infty(\Omega)}\|\nabla u\|_{L^\infty(\Omega)})$, we obtain
\begin{equation}\label{I1 ineq}
    |I_1|\leq \frac{\lambda}{4}\int_{\Omega_\delta^\prime}\tilde{a}(x,h)|D_h(\nabla u)|^2\varphi_\sigma^\delta \mathrm{d}x+C_1\int_{\Omega_\delta^\prime} \tilde{a}(x,h)\varphi_\sigma^\delta \mathrm{d}x,
\end{equation}
where $C_1$ depends on $\Omega$, $\rho^\prime$, $\nabla u$ and $\lambda$. Similarly, we infer
\begin{equation}\label{I3 ineq}
    |I_3|\leq \|\rho^\prime\|_{L^\infty(\Omega)}\|\nabla u\|_{L^\infty(\Omega)}\int_{\Omega_\delta^\prime}\tilde{a}(x,h)|\nabla \varphi_\sigma^\delta||D_h(u)|\mathrm{d}x\leq C_2 \int_{\Omega_\delta^\prime}\tilde{a}(x,h)\mathrm{d}x,
\end{equation}
where $C_2$ depends on $\Omega$, $\rho^\prime$, and $\nabla u$. Estimating $I_4$ by using Cauchy-Schwarz inequality, \eqref{partial Fi}, \eqref{grad varphi ineq} and Hölder's inequality, we obtain  
\begin{align*}
    |I_4|\leq& \|\rho\|_{L^\infty(\Omega)}\int_{\Omega_\delta^\prime}\int_0^1 |\langle F^\prime((1-z)\nabla u(x)+z\nabla u(x+h))\cdot D_h(\nabla u), \nabla \varphi_\sigma^\delta\rangle||D_h(u)|\mathrm{d}z\mathrm{d}x\\
    \leq& \Lambda \|\rho\|_{L^\infty(\Omega)}\int_{\mathcal{O}_\sigma^\delta} \tilde{a}(x,h)|D_h(\nabla u)||\nabla \varphi_\sigma^\delta||D_h(u)|\mathrm{d}x\\
    \leq& \Lambda \|\rho\|_{L^\infty(\Omega)} C^{1/2}\sigma^{-1}\int_{\mathcal{O}_\sigma^\delta} \tilde{a}(x,h)|D_h(\nabla u)|( \varphi_\sigma^\delta)^{1/2}|D_h(u)|\mathrm{d}x\\
    \leq& \Lambda \|\rho\|_{L^\infty(\Omega)} C^{1/2}\sigma^{-1}\left(\int_{\mathcal{O}_\sigma^\delta} \tilde{a}(x,h)|D_h(\nabla u)|^2 \varphi_\sigma^\delta \mathrm{d}x\right)^{1/2}\left(\int_{\mathcal{O}_\sigma^\delta}\tilde{a}(x,h)|D_h(u)|^2\mathrm{d}x\right)^{1/2}.
\end{align*}
Then, applying \eqref{Young} again with a suitable $\varepsilon$, we have
\begin{equation}\label{I4 ineq}
    |I_4|\leq \frac{\lambda}{4}\int_{\Omega_\delta^\prime}\tilde{a}(x,h)|D_h(\nabla u)|^2\varphi_\sigma^\delta \mathrm{d}x+C_3 \int_{\mathcal{O}_\sigma^\delta}\tilde{a}(x,h)\mathrm{d}x,
\end{equation}
where $C_3$ depends on $\Lambda$, $\lambda$, $\|\rho\|_{L^\infty(\Omega)}$, $\|\nabla u\|_{L^\infty(\Omega)}$ and $\sigma$. Finally, we observe that
\begin{equation}\label{us-1}
    a\int_{\Omega_\delta^\prime}D_h(|u|^{s-2}u)D_h(u)\varphi_\sigma^\delta \mathrm{d}x\leq |a|\|\nabla u\|_{L^\infty(\Omega)}\int_{\Omega_\delta^\prime} |\nabla (|u|^{s-2}u)|\mathrm{d}x<\infty,
\end{equation}
since $u \in C^{1,\alpha}(\overline{\Omega})$ and $\nabla (|u|^{s-2}u) = (s-1)|u|^{s-2}\nabla u \in L^1(\Omega) $.

Therefore, combining \eqref{I1 ineq}, \eqref{I2 ineq}, \eqref{I3 ineq}, \eqref{I4 ineq}, and \eqref{us-1} in \eqref{I1...I4}, we obtain
\begin{equation}
    \int_{\Omega_\delta'} \tilde{a}(x,h)\, |D_h(\nabla u)|^{2}\, \varphi_\sigma^\delta \, \mathrm{d}x 
    \le C < \infty.
\end{equation}
Consequently, invoking inequality~(82) in \cite{Ilyasov-Takac}, we deduce that
\begin{equation}\label{key ineq 1}
    \int_{\Omega_\delta'} \hat{a}(x;h)\, |D_h(\nabla u)|^{2}\, \varphi_\sigma^\delta \, \mathrm{d}x 
    \le C_4 < \infty,
\end{equation}
where $C_4>0$ is independent of $h$, and
\[
\hat{a}(x;h)
    := \left(
        \max_{z\in[0,1]}
        | (1-z)\nabla u(x) + z\nabla u(x+h) |
    \right)^{p-2}.
\]
By the same argument as in \cite[Theorem~3.1]{Ilyasov-Takac}, inequality~(28), we further obtain
\begin{equation}
    \int_{\Omega_\delta'} 
        |D_h(B(\nabla u))|^{2}\, \varphi_\sigma^\delta \, \mathrm{d}x
    \le
    \int_{\Omega_\delta'} 
        \hat{a}(x;h)\, |D_h(\nabla u)|^{2}\, \varphi_\sigma^\delta \, \mathrm{d}x
    \le C < \infty,
\end{equation}
where $B(\nabla u) = |\nabla u|^{p/2-1}\, \nabla u$. Consequently, it follows that $B(\nabla u) \in W^{1,2}_{\mathrm{loc}}(\Omega_\delta')$. Since $B(\eta) = |\eta|^{p/2-1} \eta$ is of class $C^1(\mathds{R}^N \setminus \{0\})$, an application of \cite[Theorem 2.1.11]{Ziemer} yields $\nabla u= B^{-1}\circ B(\nabla u) \in W^{1,2}_{\mathrm{loc}}(\Omega_\delta^\prime\cap U)$. Therefore we infer that assertion (2) holds and this completes the proof of Theorem \ref{W2regularity}.
\end{proof}


\begin{proof}[Proof of Theorem \ref{regularity a,b>0}:] If $u$ is a weak solution for \eqref{PG}, in any case we have $u \in L^\infty_{\mathrm{loc}}(\mathds{R}^N_+)$ by  \cite[Theorems 2.1, 2.2 p. 3339]{Pucci-Servadei} (using some $\theta \in (s,\infty)$ large enough instead of $p^*$ for $p=N$). Consequently, it follows from \cite[Theorem 1]{Lieberman} that $u \in C^{1,\alpha}_{\mathrm{loc}}(\mathds{R}^N_+)$ for all $\alpha\in (0,1)$. Finally, arguing as in the proof of Theorem \ref{W2regularity}, we obtain that $1$ and $2$ hold.
\end{proof}

We need to prove additional regularity results for $|\nabla u|^{p-1}$ to further deal with the case $p>2$ for nonexistence results. Precisely we have

\begin{lemma}\label{Lou} Assume $(\rho_0)$ with $\rho \in C^{1,\alpha}_{\mathrm{loc}}[0,\infty)$ for some $\alpha \in (0,1)$ and $s \in (1, p^*]$. If $u$ is a weak solution for \eqref{PG}, then
\begin{equation*}
    |\nabla u|^{p-1} \in W^{1,2}_{\mathrm{loc}}(\mathds{R}^N_+).
\end{equation*}
\end{lemma}

\begin{proof}

It suffices to prove that $ |\nabla u|^{p-1} \in W^{1,2}_{\mathrm{loc}}(\Omega) $ for any $ \Omega \subset\mathds{R}^N_+$ sufficiently smooth and bounded. From \cite[Theorems 2.1, 2.2 p. 3339]{Pucci-Servadei} (see also \cite[Theorem 1 p. 255]{Serrin}), we infer that 
\begin{equation*} 
u \in L^{\infty}_{\mathrm{loc}}(\mathds{R}^N_+)
\end{equation*}
which implies from \cite[Theorem 1]{Lieberman} that $u\in C^{1,\beta}_{\mathrm{loc}}(\mathds{R}^N_+)$, for some $\beta \in (0,1)$. Now, we adapt the proof of \cite[Lemma 2.1]{Lou}. For $ \varepsilon \in (0, 1) $, let $ f_\varepsilon \in C^\infty(\Omega) $ satisfying
\begin{equation}\label{f epsilon convergence}
\|f_\varepsilon\|_{L^r(\Omega)} \leq \|a|u|^{s-2}u\|_{L^r(\Omega)} + 1, \quad f_\varepsilon \to a|u|^{s-2}u, \text{ in } L^r(\Omega), \text{ for some large $r\geq 2$.}
\end{equation}
Let $ u_\varepsilon \in W^{1,p}(\Omega) $ be the unique solution of the equation:
\begin{equation}\label{epsilon equation}
\left\{
		\begin{aligned}
			-\mathrm{div} \left[ \rho(x_N)(\varepsilon^2 + |\nabla u_\varepsilon|^2)^{\frac{p-2}{2}} \nabla u_\varepsilon \right] &=f_\varepsilon &\mbox{in }&\ \Omega,
			\vspace{0.2cm}\\
			u_\varepsilon&=u&\mbox{on }&\
			\partial \Omega.
		\end{aligned}
		\right. 
\end{equation}
The existence of $u_\varepsilon$ can be obtained by applying \cite[Proposition 1.2, p.25]{Ekeland-Teman} to the energy functional associated with \eqref{epsilon equation}.

Since $\rho \in C^{1,\alpha}_{\mathrm{loc}}[0,\infty)$, by elliptic regularity we have for all $\beta\in (0,\alpha)$
\begin{equation}\label{W1p u epsilon convergence}
u_\varepsilon \in C^{2,\beta}(\Omega)\mbox{ and} \quad u_\varepsilon \to u, \text{ in } C^1(\overline{\Omega}), 
\end{equation}
and
\begin{equation}\label{int p/2 bounded}
\int_\Omega \rho(x_N)(\varepsilon^2 + |\nabla u_\varepsilon|^2)^{\frac{p}{2}} \, \mathrm{d}x \leq C, 
\end{equation}
here and hereafter, $ C > 0 $ denotes a constant independent of $ \varepsilon \in (0, 1) $. Again by Moser type iteration method as in \cite[Theorems 2.1 and 2.2]{Pucci-Servadei}, we get
\begin{equation}\label{u epsilon L infty estimate}
\|u_\varepsilon\|_{L^\infty(\Omega)} \leq C.
\end{equation}
Now, let
\[
G_\varepsilon(x) \equiv \frac{1}{p} (\varepsilon^2 + |x|^2)^{\frac{p}{2}}, \quad \forall x \in \mathds{R}^N,
\]
and observe that
\[
\nabla G_\varepsilon(x) = (\varepsilon^2 + |x|^2)^{\frac{p-2}{2}} x,
\]
\[
\nabla ^2 G_\varepsilon(x) = (\varepsilon^2 + |x|^2)^{\frac{p-2}{2}} I + (p - 2)(\varepsilon^2 + |x|^2)^{\frac{p-4}{2}} x \otimes x.
\]
Thus, for some $\lambda_0, \Lambda_0>0$ independent of $\varepsilon$, we have
\[
\lambda_0 |\xi|^2 \leq \frac{\langle \nabla ^2 G_\varepsilon(x) \xi, \xi \rangle}{(\varepsilon^2+|x|^2)^{\frac{p-2}{2}}} \leq \Lambda_0 |\xi|^2, \quad \forall \xi \in \mathds{R}^N.
\]
Then, we may write \eqref{epsilon equation} as
\begin{equation}\label{epsilon G equation}
\left\{
		\begin{aligned}
			-\mathrm{div} [\rho(x_N) \nabla G_\varepsilon(\nabla u_\varepsilon) ] &=f_\varepsilon &\mbox{in }&\ \Omega,
			\vspace{0.2cm}\\
			u_\varepsilon&=u&\mbox{on }&\
			\partial \Omega,
		\end{aligned}
		\right. 
\end{equation}
and, denoting $\tilde{G_\varepsilon} = \nabla ^2 G_\varepsilon(\nabla u_\varepsilon)$, for $ k = 1, 2, \dots, N$, we have
\begin{equation}\label{tilde G eq}
-\mathrm{div} \left(\rho(x_N)\tilde{G_\varepsilon} \nabla \left(\frac {\partial u_\varepsilon}{\partial x_k}\right)+\frac{\partial}{\partial x_k}(\rho(x_N))\nabla G_\varepsilon(\nabla u_\varepsilon)\right) =  \frac{\partial f_\varepsilon}{\partial x_k}, \quad \text{in } \Omega.
\end{equation}
For any domain $ \Omega_0 \subset\subset \Omega $, let $ \varphi \in C^\infty_c(\Omega) $ be such that $ 0 \leq \varphi \leq 1 $ in $ \Omega $, and $ \varphi = 1 $ in $ \Omega_0 $. Multiplying \eqref{tilde G eq} by $ \varphi^2 (\varepsilon^2 + |\nabla u_\varepsilon|^2)^{\frac{p-2}{2}} \frac {\partial u_\varepsilon}{\partial x_k} $,  integrating in $ \Omega $, and summing from $ k = 1 $ to $ N $, we get by applying the divergence theorem:
\begin{equation}\label{Ii eq}
\begin{split}
0 =& \int_\Omega \rho(x_N)\varphi^2 (\varepsilon^2 + |\nabla u_\varepsilon|^2)^{\frac{p-2}{2}} \text{tr}(\nabla ^2 u_\varepsilon \tilde{G_\varepsilon} \nabla ^2 u_\varepsilon) \, \mathrm{d}x\\
&+ (p - 2) \int_\Omega \rho(x_N)\varphi^2 (\varepsilon^2 + |\nabla u_\varepsilon|^2)^{\frac{p-4}{2}} \langle \nabla ^2 u_\varepsilon \tilde{G_\varepsilon} \nabla ^2 u_\varepsilon \nabla u_\varepsilon, \nabla u_\varepsilon \rangle \, \mathrm{d}x\\
&+ \int_\Omega 2 \rho(x_N)\varphi (\varepsilon^2 + |\nabla u_\varepsilon|^2)^{\frac{p-2}{2}} \langle \tilde{G_\varepsilon} \nabla ^2 u_\varepsilon \nabla u_\varepsilon, \nabla \varphi \rangle \, \mathrm{d}x\\
&+ \int_\Omega  \varphi^2 f_\varepsilon (\varepsilon^2 + |\nabla u_\varepsilon|^2)^{\frac{p-2}{2}}\Delta u_\varepsilon \, \mathrm{d}x\\
&+ (p - 2) \int_\Omega \varphi^2 f_\varepsilon (\varepsilon^2 + |\nabla u_\varepsilon|^2)^{\frac{p-4}{2}} \langle \nabla ^2 u_\varepsilon \nabla u_\varepsilon, \nabla u_\varepsilon \rangle \, \mathrm{d}x\\
&+ \int_\Omega 2 \varphi f_\varepsilon (\varepsilon^2 + |\nabla u_\varepsilon|^2)^{\frac{p-2}{2}} \nabla u_\varepsilon \cdot \nabla \varphi \, \mathrm{d}x\\
&+\int_{\Omega}
\rho^\prime (x_{N})\varphi^{2}
(\varepsilon^{2}+|\nabla u_\varepsilon|^{2})^{\frac{p-2}{2}}
\left\langle 
\nabla  G_\varepsilon(\nabla u_\varepsilon),\,\nabla\left(\frac {\partial u_\varepsilon}{\partial x_N} \right)
\right\rangle \mathrm{d}x\\
&+(p-2)\int_{\Omega}\rho^\prime(x_N)
\varphi^2
(\varepsilon^{2}+|\nabla u_\varepsilon|^{2})^{\frac{p-4}{2}}\left\langle 
\nabla G_\varepsilon(\nabla u_\varepsilon),\nabla^2u_\varepsilon \nabla u_\varepsilon \right\rangle \frac {\partial u_\varepsilon}{\partial x_N} \mathrm{d}x \\
&+2\int_{\Omega}\rho^\prime(x_N)
\varphi
(\varepsilon^{2}+|\nabla u_\varepsilon|^{2})^{\frac{p-2}{2}}\left\langle 
\nabla G_\varepsilon(\nabla u_\varepsilon),\nabla \varphi\right\rangle \frac {\partial u_\varepsilon}{\partial x_N} \mathrm{d}x,\\
&\equiv I_1 + I_2 + I_3 + I_4 + I_5 + I_6 +I_7+I_8+I_9.
\end{split}
\end{equation}
Arguing as in the paper \cite[Lemma 2.1]{Lou}, we have
\begin{equation}\label{est2.10}
    I_1+I_2 \geq C\lambda_0^2 \int_B \varphi^2 \, (\varepsilon^2 + |\nabla u_\varepsilon|^2)^{p-2}\|\nabla^2 u_\varepsilon\|_{tr}^2\mathrm{d}x
\end{equation}
denoting $\|A\|_{tr} \equiv \sqrt{\operatorname{tr}(A A^{T})}.$ Similarly, by \eqref{Young}, for any $\delta>0$ and some $C=C(\delta, \|\rho\|_{L^\infty(\Omega)})$, one has
\begin{equation}\label{est2.11}
    |I_3|\leq \delta\int_{\Omega}\varphi^2(\varepsilon^2+|\nabla u_\varepsilon|^2)^{p-2}\|\nabla^2 u_\varepsilon\|_{tr}^2\mathrm{d}x+\frac{C}{\delta}\int_{\Omega}(\varepsilon^2+|\nabla u_\varepsilon|^2)^{p-2}|\nabla u_\varepsilon|^2|\nabla\varphi|^2\mathrm{d}x.
\end{equation}
Since $\|f_\varepsilon\|_{L^2(\Omega)}\leq C$, we also get
\begin{equation}\label{est2.12}
|I_4+I_5|\leq \delta\int_{B}\varphi^2(\varepsilon^2+|\nabla u_\varepsilon|^2)^{p-2}\|\nabla^2u_\varepsilon\|_{tr}^2\mathrm{d}x+\frac{C}{\delta}
\end{equation}
and since $\rho\in C^1$
\begin{equation}\label{est2.13}
|I_6+I_9|\leq\int_{\Omega}(\varepsilon^2+|\nabla u_\varepsilon|^2)^{p-2}|\nabla u_\varepsilon|^2|\nabla\varphi|^2\mathrm{d}x+C.
\end{equation}
Now, taking into account again that $\rho\in C^1$, we infer that
\begin{equation}\label{estI7}
|I_7+I_8|\leq \delta\int_{\Omega}(\varepsilon^2+|\nabla u_\varepsilon|^2)^{p-2}\|\nabla^2u_\varepsilon\|_{tr}^2\mathrm{d}x+\frac{C}{\delta}\int_\Omega (\varepsilon^2+|\nabla u_\varepsilon|^2)^{p-2}|\nabla u_\varepsilon|^2\mathrm{d}x.
\end{equation}
So gathering estimates \eqref{est2.10} to \eqref{estI7}, we get 
\begin{equation}\label{est2.14}
\int_{\Omega}\varphi^2(\varepsilon^2+|\nabla u_\varepsilon|^2)^{p-2}\|\nabla^2u_\varepsilon\|_{tr}^2\leq C\int_{\Omega}(\varepsilon^2+|\nabla u_\varepsilon|^2)^{p-2}|\nabla u_\varepsilon|^2\mathrm{d}x+C.
\end{equation}
Since $(u_\varepsilon)_\varepsilon$ is uniformly bounded in $C^1(\overline{\Omega})$, the rest of the proof of \cite[Lemma 2.1 p.525-526]{Lou} can be proceeded similarly to get the result.
\end{proof}
\section{Proof of existence results}\label{secExi}

In this section, using continuous and compact embeddings results established in section \ref{secEmb}, we give the proof of existence results given in Theorems \ref{th1} and \ref{th3}. We start by proving Theorem \ref{th1} using mountain pass Theorem similarly as in \cite{NB}. We start with the following lemma:

\begin{lemma}\label{lem1}
Let $s\in(1,p_*)$, $a\leq 0$, $b>0$, $\max(q_\gamma,s)<q<p_*,$ then the functional $J:\mathcal{R}_s\to \mathds{R}$ defined by:
$$J(u)=\frac{1}{p}\|u\|^p-\frac{b}{q}\|u\|^q_{L^q(\mathds{R} ^{N-1})}-\frac{a}{s}\|u\|^s_{L^s(\mathds{R} ^{N}_+)},$$
has the mountain pass geometry.
\end{lemma}
\begin{proof}
Indeed, using Theorem \ref{improvetheo2.2}, Proposition \ref{inj N=p} in case $p=N$ and the conditions on $r,q$ we prove easily that:
\begin{itemize}
\item $J(u)\geq C_0>0$ if $\| u\| _{\mathcal{R}_s}=r_0$ for some $r_0$,
\item $\lim_{t\to \infty} J(tu)=-\infty$ for all $u\neq 0$.
\end{itemize}
\end{proof}

Note that \cite[Remark 4.1]{NB} still holds for $p\neq 2$ and $\gamma>0$. Thus we get by Ekeland Principle the existence of $\{u_n\}_{n\in\mathds{N}}\subset\mathcal{R}_s$, a nonnegative Palais-Smale sequence, that is $\{u_n\}$ in $\mathcal{R}_s$ such that $J'(u_n)\to 0, \;J(u_n)\to c$ as $n\to\infty$ where
$$c=\inf_{g\in \Lambda} \max_{t\in[0,1]} J(g(t))>0,$$
and
$$\Lambda=\{g\in C([0,1],\mathcal{R}_s):\; g(0)=0, \; J(g(1))<0\}.$$

\begin{proof}[Proof of Theorem \ref{th1}]
We start by the case $a\leq 0$, $b>0$, $s\in(1,p_*)$, $q \in (\max(q_\gamma,s),p_*)$. Let $\{u_n\}_{n\in\mathds{N}}\subset\mathcal{R}_s$ be a nonnegative Palais-Smale sequence, we show that $(u_n)$ is bounded in $\mathcal{R}_s$. Using $J'(u_n)\to 0$ and $J(u_n)\to c$ as $n\to\infty$, we get:
\begin{equation}\label{6.1}
\begin{split}
C(1+\|u_n\|+\|u_n\|_{L^s(\mathds{R} ^{N}_+)})&\geq kJ(u_n)-J'(u_n)u_n \\
&=\Big(\frac{k}{p}-1\Big)\|u_n\|^p-a\Big(\frac{k}{s}-1\Big)\|u_n\|^s_{L^s(\mathds{R} ^{N}_+)}-b\Big(\frac{k}{q}-1\Big)\|u_n\|^q_{L^q(\mathds{R} ^{N-1})}.
\end{split}
\end{equation}
Taking $k\in(\max(p,s),q)$ in \eqref{6.1} and since $a\leq 0$, we get that $\{u_n\}_{n\in\mathds{N}}$ bounded in $\mathcal{R}_s$. Then $J'(u_n)u_n\to 0$ as $n\to\infty$ and using $J(u_n)\to c$ as $n\to\infty$, we have for $n$ large enough:
$$0<k c/2\leq  kJ(u_n)-J'(u_n)u_n= \Big(\frac{k}{p}-1\Big)\|u_n\|^p-a\Big(\frac{k}{s}-1\Big)\|u_n\|^s_{L^s(\mathds{R} ^{N}_+)}-b\Big(\frac{k}{q}-1\Big)\|u_n\|^q_{L^q(\mathds{R} ^{N-1})} .$$
Taking $k<\min(p,s)$ in the above expression, we get 
\begin{equation}\label{eq}
\|u_n\|_{L^q(\mathds{R} ^{N-1})}\geq C>0.
\end{equation}
Furthermore, up to a subsequence,  $u_n\rightharpoonup u$ in $\mathcal{R}_s$ as $n\to\infty$. Thus, $J'(u_n)\to J'(u)=0$ as $n\to\infty$ and we infer that $u$ is a nonnegative weak solution to \eqref{PG}. Using \eqref{eq} and Theorem \ref{compa1}, it follows that $u$ is nontrivial. 


For the case $a>0$, $b\leq 0$, $q>1$ and $s\in (\max (s_\gamma,q),p^* )$,
similarly as 
above, we have:
$$J\,:\,u\mapsto \,J(u)=\frac{1}{p}\|u\|^p-\frac{b}{q}\|u\|^q_{L^q(\mathds{R} ^{N-1})}-\frac{a}{s}\|u\|^s_{L^s(\mathds{R} ^{N}_+)}$$
has the mountain pass geometry on $\mathcal{R}_q$ and there exists a nonnegative Palais Smale sequence $\{u_n\}_{n\in\mathds{N}}\subset\mathcal{R}_q$ and bounded, such that as $n\to\infty$:
$$J(u_n) \to c,\; J'(u_n) \to 0,$$
where the energy level $c$ verifies
$$c=\inf_{g\in \Lambda} \max_{t\in[0,1]} J(g(t))>0,$$
with
$$\Lambda=\{g\in C([0,1],\mathcal{R}_q):\; g(0)=0, \; J(g(1))<0\}.$$
We also get $\|u_n\|_{L^s(\mathds{R} ^{N}_+)}\geq C>0$ for all $n\in\mathds{N}$. Using $J'(u_n)\to J'(u)=0$ as $n\to\infty$ and Theorem \ref{compa1}, we prove that $u\geq 0$ is a nontrivial weak solution to \eqref{PG}. Finally for $\rho \in W^{1,\infty}_{\mathrm{loc}}[0,\infty)$ in both cases, using Theorem \ref{C1alpha regularity}, we get $u \in C^{1,\alpha}_{\mathrm{loc}}(\overline{\mathds{R}^N_+})\cap L^\infty(\mathds{R}^N_+)$ for all $\alpha \in (0,1)$ and $u$ verifies the boundary condition pointwisely. Furthermore, $u$ satisfies the Sobolev regularity results of Theorem \ref{W2regularity}. 
By the Harnack inequality of \cite{Trudinger} we get $u>0$ in $\mathds{R}^N_+$. Furthermore, from \cite{Vazquez}, we get $u>0$ in $\overline{\mathds{R}^N_+}$.
\end{proof}

\begin{proof}[Proof of Theorem \ref{th3}]
Let $\gamma>0$, $q\in(q_\gamma,p_*)$, $s\in(s_\gamma,p^*)$. The proof is similar to the proof of Theorem \ref{th1} and the proof of \cite[Theorem 1.7]{NB}. We only point out the main steps. The functional $J$ has the mountain pass geometry on $\mathcal{R}$ and we get a nonnegative Palais-Smale sequence $
\{u_n\}_{n\in\mathds{N}}\subset\mathcal{R}$ and bounded. Then, we show that $u_n\to u$ strongly in $\mathcal{R}$. Using $u_n$ bounded in $\mathcal{R}$, we have $\langle J'(u_n)-J'(u),u_n-u\rangle \to 0$ as $n\to\infty$ and with $u_n\to u$ in $L^s(\mathds{R} ^{N}_+)$ and $L^q(\mathds{R} ^{N-1})$ as $n\to\infty$, we get:

$$\int_{\mathds{R}^N_+}\rho(x_N)(|\nabla u_n|^{p-2}\nabla u_n -|\nabla u|^{p-2}\nabla u).(\nabla u_n-\nabla u )\mathrm{d}x \to 0.$$

Using Proposition \ref{inegalg} for $p\geq 2$ and Corollary \ref{cvtool} for $p<2$, we obtain the strong convergence in $\mathcal{R}$. From $u_n\to u$ in $\mathcal{R}$ as $n\to\infty$, we have $J(u_n)\to J(u)$ as $n\to\infty$ and then $u\neq 0$ and using $J'(u_n)\to J'(u)=0$ as $n\to\infty$, we get that $u$ is a nontrivial and nonnegative weak solution to \eqref{PG}.
For the case $\gamma>p$, when $q=p$, $s\in(p,p^*)$ we work with the norm $(\|\cdot\| ^p-b\|\cdot\|^p_{L^p(\mathds{R}^{N-1})})^{1/p}$ which is equivalent to $\|\cdot\|$ with \eqref{Hardy p} and $b/C_0<C_{p,\gamma}^{p-1}$. Similarly when $s=p$, $a/C_0<C_{p,\gamma}^p$ we work with the norm $(\|\cdot\| ^p-a\|\cdot\|^p_{L^p(\mathds{R}^{N}_+)})^{1/p}$. It is straightforward that if $\rho\in W^{1,\infty}_{\mathrm{loc}}(0,\infty)$, Theorem \ref{regularity a,b>0} holds for $u$ and using Harnack inequality \cite{Trudinger} we get $u>0$.
\end{proof}

\begin{remark}
Taking into account embeddings results in Section \ref{secEmb} for $N=p$ (see also \cite[Theorem 1.5]{abreu-felix-medeiros}), the regularity results (Theorems \ref{C1alpha regularity}, \ref{W2regularity} and \ref{regularity a,b>0}) and existence results (Theorems \ref{th1} and \ref{th3}) can be also extended to the $p$-dimension case with $p^*=p_*=\infty$ and $s,q\geq p$. Again these existence results contrast with nonexistence results in the case $\rho$ constant, see \cite[Proposition 6.1]{chipot-fila-shafrir}.
\end{remark}
 \section{Pohozaev Identity}\label{secPoho}
We prove in this section the new Pohozaev type inequality. Main ingredients are a pointwise Pohozaev identity we establish and a suitable application of Green formula.

\begin{proof}[Proof of Theorem \ref{Pohozaev}]
Let $R>0$ and $\Omega_R \subset \overline{\mathds{R}^N_+}$ be such that $B_R^+\subset\Omega_R$, $\Gamma_R\subset \partial \Omega_R$ where
    \[
    B_R^+=\{x \in \mathds{R}^N_+: |x|<R\}\quad \text{and}\quad \Gamma_R=\{x^\prime \in \mathds{R}^{N-1}: |x^\prime|\leq R\}
    \]
    and $\partial \Omega_R$ is sufficiently smooth. Since $u \in W^{2,\theta}_{\mathrm{loc}}( U) \cap C^{1}(\overline{\mathds{R}^N_+})$ (with $\theta=\min(2,p)$), we can apply the product rule in \cite[Lemma A1 and A3]{Ilyasov-Takac}  and we have the following pointwise identities :
		\begin{align*}
			{
				\rm{div}}\left((x\cdot\nabla u)\rho(x) |\nabla u|^{p-2}\nabla u\right)&=(x\cdot\nabla u)\mathrm{div}(\rho(x) |\nabla u|^{p-2}\nabla u)+\nabla(x\cdot\nabla u)\cdot\rho(x) |\nabla u|^{p-2}\nabla u\\
			&=(x\cdot\nabla u)\mathrm{div}(\rho(x) |\nabla u|^{p-2}\nabla u)+\rho(x)|\nabla u|^p+\frac{1}{p}\rho(x)\left(x\cdot\nabla (|\nabla u|^p)\right)
		\end{align*}
		and
		\begin{align*}
		{\rm{div}}\left(x\rho(x)|\nabla u|^p\right)&=\mathrm{div}(x\rho(x))|\nabla u|^p+\rho(x_N)x\cdot\nabla (|\nabla u|^p)\\
        &=\left(N\rho(x)+\langle \nabla \rho(x),x\rangle \right)|\nabla u|^p+ \rho(x)x\cdot\nabla (|\nabla u|^p),
		\end{align*}
in $\Omega_{R,0}:=\left\{x\in\Omega_R\ :\ |\nabla u(x)|\neq0\right\}$. Then we obtain the following pointwise Pohozaev identity:
\begin{equation*}
    \mathrm{div}\left((x\cdot\nabla u)\rho(x) |\nabla u|^{p-2}\nabla u - \frac{1}{p}x\rho(x)|\nabla u|^p\right)= -(x\cdot\nabla u)f(u)+\frac{p-N}{p}\rho(x)|\nabla u|^p-\frac{1}{p}\langle \nabla \rho(x),x\rangle |\nabla u|^p,
\end{equation*}
in $\Omega_{R,0}$, where we have implemented $-\mathrm{div}(\rho(x)|\nabla u|^{p-2}\nabla u)=f(u)$ in $\Omega_{R,0}$. Now, consider $\psi_R \in C_0^\infty(B_R)$ such that
        \begin{align*}
			\psi_{R}(x)=\left\{
			\begin{aligned}
				1&\quad&\text{if}&\quad |x|\leq R/2,\\
				0&\quad&\text{if}&\quad |x|\geq R,
			\end{aligned}
			\right.
		\end{align*}
with $|\nabla \psi_{R}(x)|\leq2/R$, and define the vector field $\mathbf{v}_R: \Omega_R \longrightarrow\mathds{R}^N$ by
		\begin{equation*}
		   \mathbf{v}_{R}(x):=\left[(x\cdot\nabla u)\rho(x) |\nabla u|^{p-2}\nabla u - \frac{1}{p}x\rho(x)|\nabla u|^p\right]\psi_{R}.
		\end{equation*}
From the previous computations, it follows that		
\begin{align}\label{div vr}
		    \nonumber-\mathrm{div}(\mathbf{v}_R)=&\left[(x\cdot\nabla u)f(u)+\frac{N-p}{p}\rho(x)|\nabla u|^p+\frac{1}{p}\langle \nabla \rho(x),x\rangle|\nabla u|^p\right]\psi_R\\
            -&\left[(x\cdot\nabla u)\rho(x) |\nabla u|^{p-2}\nabla u - \frac{1}{p}x\rho(x)|\nabla u|^p\right]\nabla\psi_R
		\end{align}
in $\Omega_{R,0}$. 

\textbf{Claim:} \eqref{div vr} holds in the distributional sense in $\Omega_R$. 

We follow the proof of \cite[Lemma 4.5]{Ilyasov-Takac}: First consider for $\varepsilon>0$ the sets
\begin{equation*}
    \Omega_{R,\varepsilon}:=\left\{x\in\Omega_R\ :\ |\nabla u(x)|>\varepsilon\right\}\quad \text{and}\quad \Omega_{R,\varepsilon}^\prime:=\left\{x\in\Omega_R\ :\ |\nabla u(x)|<\varepsilon\right\}= \Omega_R\setminus \overline{\Omega_{R,\varepsilon}},
\end{equation*}
and the Urysohn-type function $\zeta_\varepsilon: \Omega_R \longrightarrow \mathds{R}$, given by $\zeta_\varepsilon=\min\{\varepsilon^{-r}|\nabla u|^r,1\}$, with $r \in (1,p)$, that is,
\begin{equation*}
    \zeta_\varepsilon= \left\{
			\begin{aligned}
				\varepsilon^{-r}|\nabla u|^r &\quad&\text{in}&\quad \Omega_{R,\varepsilon}^\prime,\\
				1 &\quad&\text{in}&\quad \overline{\Omega_{R,\varepsilon}}.
			\end{aligned}
			\right.
\end{equation*}
Observe that $\zeta_\varepsilon \in W^{1,1}_{\mathrm{loc}}(\Omega_R)$, since $u$ satisfies $|\nabla u|^r \in W^{1,1}_{\mathrm{loc}}(\Omega_R)$ and $u \in C^{1}(\overline{\Omega_R})$. Precisely, 
\begin{equation}\label{grad zeta}
    |\nabla \zeta_\varepsilon|\leq r\varepsilon^{-r}|\nabla u|^{r-1}|\nabla^2u| \in L^1_{\mathrm{loc}}(\Omega_R).
\end{equation}
Then, given $\varphi \in C^1_0(\Omega_R)$, we can write $\varphi= \zeta_\varepsilon \varphi+(1-\zeta_\varepsilon)\varphi$ with
\begin{equation*}
    \zeta_\varepsilon \varphi \in W^{1,1}_0(\Omega_{R,0})\quad\text{and}\quad (1-\zeta_\varepsilon)\varphi \in W^{1,1}_0(\Omega_{R,\varepsilon}^\prime).
\end{equation*}
Using \eqref{div vr} and \cite[Lemma A3]{Ilyasov-Takac}, we have
\begin{align}\label{distrib sense OmegaR0}
			\nonumber\int_{\Omega_R}\mathbf{v}_{_R}\cdot \nabla \phi\,\mathrm{d}x &=\int_{\Omega_R}\left[\left((x\cdot\nabla u)f(u)+\frac{N-p}{p}\rho(x)|\nabla u|^p+\frac{1}{p}\langle \nabla \rho(x),x\rangle|\nabla u|^p\right)\psi_{R}\right.\\
			&-\left.\left((x\cdot\nabla u)\rho(x) |\nabla u|^{p-2}\nabla u - \frac{1}{p}x\rho(x)|\nabla u|^p\right)\nabla\psi_{_R}\right]\phi\,\mathrm{d}x,
		\end{align}
for all $\phi \in W^{1,1}_0(\Omega_{R,0})$. In particular, \eqref{distrib sense OmegaR0} holds with $\zeta_\varepsilon\varphi$ in place of $\phi$. Then, we have
\begin{align}\label{vr varphi}
\nonumber \int_{\Omega_R}\mathbf{v}_{_R}\cdot \nabla \varphi \mathrm{d}x=& \int_{\Omega_R}\mathbf{v}_{_R}\cdot \nabla (\zeta_\varepsilon\varphi)\,\mathrm{d}x+\int_{\Omega_R}\mathbf{v}_{_R}\cdot \nabla ((1-\zeta_\varepsilon)\varphi)\,\mathrm{d}x\\
\nonumber=& \int_{\Omega_R}\left[\left((x\cdot\nabla u)f(u)+\frac{N-p}{p}\rho(x)|\nabla u|^p+\frac{1}{p}\langle \nabla \rho(x),x\rangle|\nabla u|^p\right)\psi_{R}\right.\\
			&-\left.\left((x\cdot\nabla u)\rho(x) |\nabla u|^{p-2}\nabla u - \frac{1}{p}x\rho(x)|\nabla u|^p\right)\nabla\psi_{_R}\right](\zeta_\varepsilon\varphi)\,\mathrm{d}x+\int_{\Omega_R}\mathbf{v}_{_R}\cdot \nabla ((1-\zeta_\varepsilon)\varphi)\,\mathrm{d}x.
\end{align}
Now, observe that $|\mathbf{v}_R|\leq C_p \|\rho\|_{L^\infty(\Omega_R)}|x||\nabla u|^p$ and thus as $\varepsilon \longrightarrow 0^+$
\begin{align*}
    \left|\int_{\Omega_R}\mathbf{v}_R\cdot \nabla((1-\zeta_\varepsilon)\varphi)\mathrm{d}x\right|\leq& C_p \int_{\Omega_{R,\varepsilon}^\prime}|x||\nabla u|^p\left(|\nabla \varphi|+|\varphi||\nabla \zeta_\varepsilon|\right)\mathrm{d}x\\
    =& C_p \varepsilon^{p-r}\int_{\Omega_{R,\varepsilon}^\prime}|x|\left(\frac{|\nabla u|}{\varepsilon}\right)^p\varepsilon^r(|\varphi||\nabla \zeta_\varepsilon|+|\nabla \varphi|)\mathrm{d}x\\
    \leq& C_p \varepsilon^{p-r}\int_{\Omega_{R,\varepsilon}^\prime}|x|\varepsilon^r(|\varphi||\nabla \zeta_\varepsilon|+|\nabla \varphi|)\mathrm{d}x\\
    \longrightarrow& 0,
\end{align*}
since $r \in (1,p)$ and the integral above is bounded as a consequence of \eqref{grad zeta}. In the other hand, note that as $\varepsilon \longrightarrow 0^+$,
\begin{equation*}
    \zeta_\varepsilon \longrightarrow \zeta_0= \left\{
			\begin{aligned}
				1&\quad&\text{in}&\quad \Omega_{R,0},\\
				0&\quad&\text{in}&\quad \Omega_{R,0}^\prime,
			\end{aligned}
			\right.
\end{equation*}
 where $\Omega_{R,0}^\prime = \{x \in \Omega_R: \nabla u(x)=0\}=\Omega_R\setminus \Omega_{R,0}$. Therefore by dominated convergence theorem, passing to the limit in \eqref{vr varphi} as $\varepsilon \longrightarrow 0^+$, we conclude the claim.

Applying the divergence theorem \cite[Lemma A.1]{Ilyasov-Takac}, we have
\begin{align*}
\int_{\Gamma_R}\mathbf{v}_{R}\cdot\nu\,\mathrm{d}x^\prime            &=-\int_{ \Omega_R}\left[(x\cdot\nabla u)f(u)+\frac{N-p}{p}\rho(x)|\nabla u|^p+\frac{1}{p}\langle \nabla \rho(x),x\rangle|\nabla u|^p\right]\psi_R\,\mathrm{d}x\\ 
			&+\int_{B_{R}^+\setminus B_{R/2}}\left[(x\cdot\nabla u)\rho(x)|\nabla u|^{p-2}\nabla u - \frac{1}{p}x\rho(x)|\nabla u|^p\right]\nabla\psi_{R}\,\mathrm{d}x.
\end{align*}
Given that $|\nabla \psi_{R}(x)|\leq2/R$ and $\rho(x)|\nabla u|^p\in L^1(\mathds{R}^N_+)$, we have, as $R\longrightarrow+\infty$,
		\begin{align*}
			\left|\int_{B_{R}^+\setminus B_{R/2}}\left[(x\cdot\nabla u)\rho(x) |\nabla u|^{p-2}\nabla u - \frac{1}{p}x\rho(x)|\nabla u|^p\right]\nabla\psi_{R}\,\mathrm{d}x\right|          
			&\leq 2\frac{p+1}{p}\int_{B_{R}^+\setminus B_{R/2}}\rho(x)|\nabla u|^p\,\mathrm{d}x\longrightarrow 0.
		\end{align*}
		 Using the boundary condition, the definition of $\mathbf{v}_R$, and that $x\cdot \nu=0$ in $\Gamma_{R}$, we obtain
		\begin{align}\label{limite}
		\nonumber \int_{\Gamma_{R}}g(u)(x\cdot\nabla u) \psi_R\,\mathrm{d}x^\prime=&\int_{\Gamma_{R}}\mathbf{v}_R \cdot \nu \mathrm{d}x^\prime\\
            =& -\int_{ \Omega_R}\left[(x\cdot\nabla u)f(u)+\frac{N-p}{p}\rho(x)|\nabla u|^p+\frac{1}{p}\langle \nabla \rho(x),x\rangle|\nabla u|^p\right]\psi_R\,\mathrm{d}x+o_R(1).
		\end{align}
		We now claim that
        \begin{equation}\label{f limit}
            \int_{\Omega_{R}}(x\cdot \nabla u)f(u)\psi_R\mathrm{d}x\longrightarrow -N\int_{\mathds{R}^N_+}F(u)\mathrm{d}x,\quad \text{as}\quad R\longrightarrow\infty.
        \end{equation}
        Indeed, integrating by parts, we obtain 
		\begin{align*}
			\int_{\Omega_R}f(u)(x\cdot\nabla u)\psi_R\,\mathrm{d}x	&=\sum_{i=1}^N\int_{\Omega_R}\left(F(u)\right)_{x_i}x_i\psi_R\,\mathrm{d}x=\int_{\Omega_R}\left[-NF(u)\psi_R-F(u)(x\cdot\nabla\psi_R)\right]\,\mathrm{d}x.
		\end{align*}
		By dominated convergence theorem,
		$$
\int_{\Omega_R}F(u)\psi_R\,\mathrm{d}x\longrightarrow\int_{\mathds{R}^N_+}F(u)\,\mathrm{d}x, \quad \text{as}\quad R\longrightarrow\infty.
		$$
Furthermore
		\begin{align*}
		\left|\int_{\Omega_R}F(u)(x\cdot\nabla\psi_R)\,\mathrm{d}x\right|&\leq \int_{B_{R}^+\backslash B_{R/2}}|F(u)||x||\nabla\psi_R|\,\mathrm{d}x\leq 2\int_{B_{R}^+\backslash B_{R/2}}|F(u)|\,\mathrm{d}x\longrightarrow0,\quad \text{as}\quad R\longrightarrow\infty,   
		\end{align*}
since $F(u)\in L^1(\mathds{R}^N_+)$ and this proves \eqref{f limit}. Now, observe that 
		$$
		\begin{aligned}
		\int_{\Gamma_{R}}g(u)(x\cdot\nabla u) \psi_R\,\mathrm{d}x^\prime=\sum_{i=1}^{N-1}\int_{\Gamma_{R}}\left(G(u)\right)_{x_i}x_i\psi_R\,\mathrm{d}x^\prime=\int_{\Gamma_{R}}\left[-(N-1)G(u)\psi_R-G(u)\left(x^\prime\cdot
		\nabla_{x^\prime}\psi_R\right)\right]\,\mathrm{d}x^\prime.
		\end{aligned}
		$$
Arguing as in the derivation of \eqref{f limit}, we obtain, as $R\longrightarrow+\infty$,
\begin{equation}\label{g limit}
    \int_{\Gamma_{R}}g(u)(x\cdot\nabla u) \psi_R\,\mathrm{d}x^\prime \longrightarrow -(N-1)\int_{\mathds{R}^{N-1}}G(u)\mathrm{d}x^\prime.
\end{equation}
		Finally, since $\rho(x)|\nabla u|^p \in L^1(\mathds{R}^N_+)$, by \eqref{rho1} and dominated convergence theorem we have, as $R\longrightarrow+\infty$,
		\begin{equation}\label{rho limit}
		\int_{\Omega_R}\rho(x)|\nabla u|^p\psi_R\,\mathrm{d}x\longrightarrow\int_{\mathds{R}^N_+}\rho(x)|\nabla u|^p\,\mathrm{d}x\quad \text{and}\quad \int_{\Omega_R}\langle \nabla \rho(x),x\rangle|\nabla u|^p\psi_R\,\mathrm{d}x\longrightarrow\int_{\mathds{R}^N_+}\langle \nabla \rho(x),x\rangle|\nabla u|^p\,\mathrm{d}x.
		\end{equation}       
Passing to the limit in \eqref{limite} as $R\longrightarrow+\infty$ and invoking \eqref{f limit}, \eqref{g limit} and \eqref{rho limit}, the desired identity follows, and the proof is complete.
\end{proof}

\section{Proof of nonexistence results}\label{secNon}



We only prove here Theorem \ref{Nonexistence1}. The proof of Theorems \ref{Nonexistence2} and \ref{nonexPleqN} is similar by assuming the regularity of weak solutions requested to  apply Pohozaev identity (Theorem \ref{Pohozaev}).

\begin{proof}[Proof of Theorem \ref{Nonexistence1}]

Let $u$ be a 
weak solution. In any case, by Theorems \ref{C1alpha regularity}, we have $u \in C^{1,\alpha}_{\mathrm{loc}}(\overline{\mathds{R}^N_+})$. If $p \in (1,2]$, by Theorem \ref{W2regularity}, $u \in W^{2,p}_{\mathrm{loc}}(\mathds{R}^N_+)$ and hence $|\nabla u|^r \in W^{1,1}_{\mathrm{loc}}(\mathds{R}^N_+)$, for $r\in (1,p)$. On the other hand, if $p>2$, $u \in W^{2,2}_{\mathrm{loc}}(U)$, where $U$ is as defined in Theorem \ref{W2regularity}, and by Lemma \ref{Lou}, $|\nabla u|^{p-1} \in W^{1,1}_{\mathrm{loc}}(\mathds{R}^N_+)$. Then, the Pohozaev identity in Theorem \ref{Pohozaev} holds and yields:
\begin{align*}
     \frac{N-p}{p}\int_{\mathds{R}_{+}^{N}}\rho(x_N)|\nabla u|^p\mathrm{d}x+\frac{1}{p}\int_{\mathds{R}_{+}^{N}}\rho'(x_N)x_N|\nabla u|^p\mathrm{d}x=&\frac{aN}{s}\int_{\mathds{R}^N_+}|u|^s\mathrm{d}x\\
     +&\frac{b(N-1)}{q}\int_{\mathds{R}^{N-1}}|u|^q\mathrm{d}x^\prime.
\end{align*}
Furthermore, taking $u$ as a testing function, we obtain
\begin{equation*}
    \int_{\mathds{R}^N_+}\rho(x_N)|\nabla u|^p\mathrm{d}x=a\int_{\mathds{R}^N_+}|u|^s\mathrm{d}x+b\int_{\mathds{R}^{N-1}}|u|^q\mathrm{d}x^\prime.
\end{equation*}
By combining the above two equalities, we get the following relation:
\begin{align*}
    \frac{1}{p}\int_{\mathds{R}^N_+}\rho'(x_N)x_N|\nabla u|^p\mathrm{d}x&=a\left(\frac{N}{s}-\frac{N-p}{p}\right)\int_{\mathds{R}^{N}_+}|u|^s\mathrm{d}x\\
    &+b\left(\frac{N-1}{q}-\frac{N-p}{p}\right)\int_{\mathds{R}^{N-1}}|u|^q\mathrm{d}x^\prime.
\end{align*}
Since $\rho^\prime(x_N)x_N>0$ by $(\rho_1)$, the left-hand side of the equation is strictly positive unless $u\equiv 0$. Under conditions of items $(i)$ or $(ii)$,  the right-hand side becomes nonpositive. Therefore, the only possibility is $u\equiv 0$, which completes the proof.
\end{proof}

\begin{remark} When $p \in (1,2]$ we can obtain the result above with less regularity: $\rho \in W^{1,\infty}_{\mathrm{loc}}(\mathds{R}^+)$, since Lemma \ref{Lou} is not required in the proof of Pohazaev identity.    
\end{remark}


\section{Appendix}

We recall some well known technical results (see e.g. \cite{simon}):
\begin{proposition}\label{inegalg}
There exist $c$ a positive constant such that for any $\xi , \eta \in \mathds{R}^d$:

\begin{align}
(|\xi|^{p-2}\xi -|\eta|^{p-2}\eta).(\xi-\eta) \geq c \left\{ \begin{array}{lr}  |\xi -\eta |^{p} & \text{if } p\geq 2, \\
\dfrac{|\xi-\eta|^{2}}{(|\xi|+|\eta|)^{2-p}} & \text{if } p \leq 2.
\end{array} \right.
\end{align}

\end{proposition}

Using H\"older inequality, we obtain easily:

\begin{corollary}\label{cvtool}
Let $p<2$ then there exists $C>0$ such that for $u,v\in \mathcal{D}^{1,p}_\rho(\mathds{R}^N_+) $:

$$\| u+v\|^{2-p}\int_{\mathds{R}^N_+}\rho(x_N)(|\nabla u|^{p-2}\nabla u -|\nabla v|^{p-2}\nabla v).(\nabla u-\nabla v )  \geq  C \| u-v \| ^2.$$
\end{corollary}

\noindent\textbf{Funding.}
This study was financed in part by the Coordenação de Aperfeiçoamento de Pessoal de Nível Superior - Brasil (CAPES) - Finance Code 001, and by the MathAmSud grant 88887.878894/2023-00.

\noindent Data Availability Statement: The manuscript has no associated data.\\
     
\noindent Conflict of Interest Statement : The authors declare that they have no conflict of interest.

\bibliographystyle{plain}
\bibliography{biblio}

@article {CAOM,
    AUTHOR = {Abreu, E. and Clemente, R. and Do \'O, J. M. and Medeiros, E.},
     TITLE = {{$p$}-harmonic functions in the upper half-space},
   JOURNAL = {Potential Anal.},
  FJOURNAL = {Potential Analysis. An International Journal Devoted to the
              Interactions between Potential Theory, Probability Theory,
              Geometry and Functional Analysis},
    VOLUME = {60},
      YEAR = {2024},
    NUMBER = {4},
     PAGES = {1383--1406},
      ISSN = {0926-2601,1572-929X},
   MRCLASS = {35J66 (35B09 35B35 35J62 35J70)},
  MRNUMBER = {4732429},
MRREVIEWER = {Rosa\ Pardo},
       DOI = {10.1007/s11118-023-10092-7},
}

@article {abreu-felix-medeiros,
    AUTHOR = {Abreu, E. and Felix, D. D. and Medeiros, E.},
     TITLE = {A weighted {H}ardy type inequality and its applications},
   JOURNAL = {Bull. Sci. Math.},
  FJOURNAL = {Bulletin des Sciences Math\'ematiques},
    VOLUME = {166},
      YEAR = {2021},
     PAGES = {Paper No. 102937, 25},
      ISSN = {0007-4497,1952-4773},
   MRCLASS = {46E35 (35A01 35B09 35C06 35J65 35J66 35K05)},
  MRNUMBER = {4189791},
MRREVIEWER = {Flavia\ Giannetti},
}

@article {Abreu-Furtado-Medeiros2,
    AUTHOR = {Abreu, E. A. M. and Furtado, M. F. and Medeiros, E.},
     TITLE = {On a {H}ardy-{S}obolev inequality with remainder term and its
              consequences},
   JOURNAL = {J. Geom. Anal.},
  FJOURNAL = {Journal of Geometric Analysis},
    VOLUME = {35},
      YEAR = {2025},
    NUMBER = {12},
     PAGES = {Paper No. 395, 25},
      ISSN = {1050-6926,1559-002X},
   MRCLASS = {35A23 (26D10)},
  MRNUMBER = {4977134},
       DOI = {10.1007/s12220-025-02235-6},
}

@article {Abreu-JM-Medeiros,
    AUTHOR = {Abreu, E. and Do \'O, J.M. and Medeiros, E.},
     TITLE = {Properties of positive harmonic functions on the half-space
              with a nonlinear boundary condition},
   JOURNAL = {J. Differential Equations},
  FJOURNAL = {Journal of Differential Equations},
    VOLUME = {248},
      YEAR = {2010},
    NUMBER = {3},
     PAGES = {617--637},
      ISSN = {0022-0396,1090-2732},
   MRCLASS = {35J65 (35B40 35B45 35J20)},
  MRNUMBER = {2557909},
MRREVIEWER = {Joanna\ Pres-Jennings},
       DOI = {10.1016/j.jde.2009.07.006},
}

@article{AMY,
AUTHOR = {Abreu, E.A.M. and Medeiros,
              E.S. and Yang, J.},
     TITLE = {On a weighted Sobolev embedding on the upper half-space in the borderline case},
      JOURNAL = {Ann. Mat. Pura Appl. (4)},
  FJOURNAL = {Annali di Matematica Pura ed Applicata. Series IV},
    VOLUME = {201},
      YEAR = {2022},
    NUMBER = {6},
     PAGES = {2715--2732},
      ISSN = {0373-3114,1618-1891},
   MRCLASS = {46E35 (35J66)},
  MRNUMBER = {4498406},
MRREVIEWER = {Maria\ A.\ Ragusa},
       DOI = {10.1007/s10231-022-01217-7},
       URL = {https://doi-org.rproxy.univ-pau.fr/10.1007/s10231-022-01217-7},
}

@article {Berestycki-Lions,
    AUTHOR = {Berestycki, H. and Lions, P.-L.},
     TITLE = {Nonlinear scalar field equations. {I}. {E}xistence of a ground
              state},
   JOURNAL = {Arch. Rational Mech. Anal.},
  FJOURNAL = {Archive for Rational Mechanics and Analysis},
    VOLUME = {82},
      YEAR = {1983},
    NUMBER = {4},
     PAGES = {313--345},
      ISSN = {0003-9527},
   MRCLASS = {35J60 (35Q20 58E99 81E99)},
  MRNUMBER = {695535},
MRREVIEWER = {Wei\ Ming\ Ni},
       DOI = {10.1007/BF00250555},
}

@article {DeThelin,
    AUTHOR = {de Th\'elin, F.},
     TITLE = {Local regularity properties for the solutions of a nonlinear
              partial differential equation},
   JOURNAL = {Nonlinear Anal.},
  FJOURNAL = {Nonlinear Analysis. Theory, Methods \& Applications. An
              International Multidisciplinary Journal},
    VOLUME = {6},
      YEAR = {1982},
    NUMBER = {8},
     PAGES = {839--844},
      ISSN = {0362-546X,1873-5215},
   MRCLASS = {35J60 (35B65)},
  MRNUMBER = {671725},
MRREVIEWER = {Pavel\ Doktor},
       DOI = {10.1016/0362-546X(82)90068-2},
}

@book {Ekeland-Teman,
    AUTHOR = {Ekeland, I. and T\'emam, R.},
     TITLE = {Convex analysis and variational problems},
    SERIES = {Classics in Applied Mathematics},
    VOLUME = {28},
   EDITION = {English},
      NOTE = {Translated from the French},
 PUBLISHER = {SIAM},
 FPUBLISHER = {Society for Industrial and Applied Mathematics (SIAM),
              Philadelphia, PA},
      YEAR = {1999},
     PAGES = {xiv+402},
      ISBN = {0-89871-450-8},
   MRCLASS = {49-02 (01A75 49J53 90C46)},
  MRNUMBER = {1727362},
       DOI = {10.1137/1.9781611971088},
       URL = {https://doi.org/10.1137/1.9781611971088},
}

@article {Escobar1,
    AUTHOR = {Escobar, J. F.},
     TITLE = {Sharp constant in a {S}obolev trace inequality},
   JOURNAL = {Indiana Univ. Math. J.},
  FJOURNAL = {Indiana University Mathematics Journal},
    VOLUME = {37},
      YEAR = {1988},
    NUMBER = {3},
     PAGES = {687--698},
      ISSN = {0022-2518,1943-5258},
   MRCLASS = {46E35},
  MRNUMBER = {962929},
MRREVIEWER = {Philippe-A.\ Dionne},
       DOI = {10.1512/iumj.1988.37.37033},
}

@article {Escobar2,
    AUTHOR = {Escobar, J. F.},
     TITLE = {Uniqueness theorems on conformal deformation of metrics,
              {S}obolev inequalities, and an eigenvalue estimate},
   JOURNAL = {Comm. Pure Appl. Math.},
  FJOURNAL = {Communications on Pure and Applied Mathematics},
    VOLUME = {43},
      YEAR = {1990},
    NUMBER = {7},
     PAGES = {857--883},
      ISSN = {0010-3640,1097-0312},
   MRCLASS = {58E11 (35P15 53C21 58G30)},
  MRNUMBER = {1072395},
       DOI = {10.1002/cpa.3160430703},
}

@article{NB,
  title={Existence and Nonexistence Breaking Results For a Weighted Elliptic Problem in Half-Space},
  author={Do Ó, J.M and Freire, R.F. and Giacomoni, J. and Medeiros, E.S.},
  journal={arXiv preprint arXiv:2510.05999},
  year={2025}
}

@book {GT,
    AUTHOR = {Gilbarg, D. and Trudinger, N. S.},
     TITLE = {Elliptic partial differential equations of second order},
    SERIES = {Classics in Mathematics},
      NOTE = {Reprint of the 1998 edition},
 PUBLISHER = {Springer-Verlag, Berlin},
      YEAR = {2001},
     PAGES = {xiv+517},
      ISBN = {3-540-41160-7},
   MRCLASS = {35-02 (35Jxx)},
  MRNUMBER = {1814364},
}

@book {HardyLittlewood,
    AUTHOR = {Hardy, G. H. and Littlewood, J. E. and P\'olya, G.},
     TITLE = {Inequalities},
      NOTE = {2d ed},
 PUBLISHER = {Cambridge, at the University Press, },
      YEAR = {1952},
     PAGES = {xii+324},
   MRCLASS = {27.0X},
  MRNUMBER = {46395},
}

@article {Ilyasov-Takac,
    AUTHOR = {Il'yasov, Y. S. and Tak\'a\v{c}, P.},
     TITLE = {Optimal {$W^{2,2}_{\rm loc}$}-regularity, Pohozhaev's
              identity, and nonexistence of weak solutions to some
              quasilinear elliptic equations},
   JOURNAL = {J. Differential Equations},
  FJOURNAL = {Journal of Differential Equations},
    VOLUME = {252},
      YEAR = {2012},
    NUMBER = {3},
     PAGES = {2792--2822},
      ISSN = {0022-0396,1090-2732},
   MRCLASS = {35J92 (35A01 35B65 35D30 35J20 35J70 35J75)},
  MRNUMBER = {2860641},
MRREVIEWER = {Alan\ V.\ Lair},
       DOI = {10.1016/j.jde.2011.10.020},
}

@book {Kawohl,
    AUTHOR = {Kawohl, B.},
     TITLE = {Rearrangements and convexity of level sets in {PDE}},
    SERIES = {Lecture Notes in Mathematics},
    VOLUME = {1150},
 PUBLISHER = {Springer-Verlag, Berlin},
      YEAR = {1985},
     PAGES = {iv+136},
      ISBN = {3-540-15693-3},
   MRCLASS = {35-02 (35B50 35J60 49A50)},
  MRNUMBER = {810619},
MRREVIEWER = {Michael\ Wiegner},
       DOI = {10.1007/BFb0075060},
}

@article {Li-Zhu,
    AUTHOR = {Li, Y. and Zhu, M.},
     TITLE = {Uniqueness theorems through the method of moving spheres},
   JOURNAL = {Duke Math. J.},
  FJOURNAL = {Duke Mathematical Journal},
    VOLUME = {80},
      YEAR = {1995},
    NUMBER = {2},
     PAGES = {383--417},
      ISSN = {0012-7094,1547-7398},
   MRCLASS = {35J65 (35B99 35K60)},
  MRNUMBER = {1369398},
MRREVIEWER = {Alan\ V.\ Lair},
       DOI = {10.1215/S0012-7094-95-08016-8},
}

@article {YanYan-Lei,
    AUTHOR = {Li, Y. and Zhang, L.},
     TITLE = {Liouville-type theorems and {H}arnack-type inequalities for
              semilinear elliptic equations},
   JOURNAL = {J. Anal. Math.},
  FJOURNAL = {Journal d'Analyse Math\'ematique},
    VOLUME = {90},
      YEAR = {2003},
     PAGES = {27--87},
      ISSN = {0021-7670,1565-8538},
   MRCLASS = {35J60 (35B33 35B40 35B45 35B65)},
  MRNUMBER = {2001065},
MRREVIEWER = {Luisa\ Moschini},
       DOI = {10.1007/BF02786551},
}

@article {Lieberman,
    AUTHOR = {Lieberman, G. M.},
     TITLE = {Boundary regularity for solutions of degenerate elliptic
              equations},
   JOURNAL = {Nonlinear Anal.},
  FJOURNAL = {Nonlinear Analysis. Theory, Methods \& Applications. An
              International Multidisciplinary Journal},
    VOLUME = {12},
      YEAR = {1988},
    NUMBER = {11},
     PAGES = {1203--1219},
      ISSN = {0362-546X,1873-5215},
   MRCLASS = {35J70 (35B65)},
  MRNUMBER = {969499},
MRREVIEWER = {Zuchi\ Chen},
       DOI = {10.1016/0362-546X(88)90053-3},
}

@article {PLLions,
    AUTHOR = {Lions, P. L.},
     TITLE = {Sym\'etrie et compacit\'e{} dans les espaces de {S}obolev},
   JOURNAL = {J. Functional Analysis},
  FJOURNAL = {Journal of Functional Analysis},
    VOLUME = {49},
      YEAR = {1982},
    NUMBER = {3},
     PAGES = {315--334},
      ISSN = {0022-1236},
   MRCLASS = {46E35},
  MRNUMBER = {683027},
       DOI = {10.1016/0022-1236(82)90072-6},
}

@article {Lou,
    AUTHOR = {Lou, H.},
     TITLE = {On singular sets of local solutions to {$p$}-{L}aplace
              equations},
   JOURNAL = {Chinese Ann. Math. Ser. B},
  FJOURNAL = {Chinese Annals of Mathematics. Series B},
    VOLUME = {29},
      YEAR = {2008},
    NUMBER = {5},
     PAGES = {521--530},
      ISSN = {0252-9599,1860-6261},
   MRCLASS = {35J60 (35A20 35J70 49J45)},
  MRNUMBER = {2447484},
MRREVIEWER = {Gabriella\ Bogn\'ar},
       DOI = {10.1007/s11401-007-0312-y},
}

@article {Pohozaev,
    AUTHOR = {Poho\v zaev, S. I.},
     TITLE = {On the eigenfunctions of the equation {$\Delta u+\lambda
              f(u)=0$}},
   JOURNAL = {Dokl. Akad. Nauk SSSR},
  FJOURNAL = {Doklady Akademii Nauk SSSR},
    VOLUME = {165},
      YEAR = {1965},
     PAGES = {36--39},
      ISSN = {0002-3264},
   MRCLASS = {35.80},
  MRNUMBER = {192184},
MRREVIEWER = {O.\ Em.\ Gheorghiu},
}

@article {Nazaret,
    AUTHOR = {Nazaret, B.},
     TITLE = {Best constant in {S}obolev trace inequalities on the
              half-space},
   JOURNAL = {Nonlinear Anal.},
  FJOURNAL = {Nonlinear Analysis. Theory, Methods \& Applications. An
              International Multidisciplinary Journal},
    VOLUME = {65},
      YEAR = {2006},
    NUMBER = {10},
     PAGES = {1977--1985},
      ISSN = {0362-546X,1873-5215},
   MRCLASS = {46E35 (35J20 35J60)},
  MRNUMBER = {2258478},
MRREVIEWER = {Eduardo\ Colorado},
       DOI = {10.1016/j.na.2005.05.069},
}

@article {Pucci-Servadei,
    AUTHOR = {Pucci, P. and Servadei, R.},
     TITLE = {Regularity of weak solutions of homogeneous or inhomogeneous
              quasilinear elliptic equations},
   JOURNAL = {Indiana Univ. Math. J.},
  FJOURNAL = {Indiana University Mathematics Journal},
    VOLUME = {57},
      YEAR = {2008},
    NUMBER = {7},
     PAGES = {3329--3363},
      ISSN = {0022-2518,1943-5258},
   MRCLASS = {35J60 (35D10)},
  MRNUMBER = {2492235},
MRREVIEWER = {Diego\ M.\ Maldonado},
       DOI = {10.1512/iumj.2008.57.3525},
}

@article {Serrin,
    AUTHOR = {Serrin, J.},
     TITLE = {Local behavior of solutions of quasi-linear equations},
   JOURNAL = {Acta Math.},
  FJOURNAL = {Acta Mathematica},
    VOLUME = {111},
      YEAR = {1964},
     PAGES = {247--302},
      ISSN = {0001-5962,1871-2509},
   MRCLASS = {35.47},
  MRNUMBER = {170096},
MRREVIEWER = {C.\ B.\ Morrey, Jr.},
       DOI = {10.1007/BF02391014},
}

@article {Trudinger,
    AUTHOR = {Trudinger, N. S.},
     TITLE = {On {H}arnack type inequalities and their application to
              quasilinear elliptic equations},
   JOURNAL = {Comm. Pure Appl. Math.},
  FJOURNAL = {Communications on Pure and Applied Mathematics},
    VOLUME = {20},
      YEAR = {1967},
     PAGES = {721--747},
      ISSN = {0010-3640,1097-0312},
   MRCLASS = {35.47},
  MRNUMBER = {226198},
MRREVIEWER = {A.\ Kufner},
       DOI = {10.1002/cpa.3160200406},
}

@article {Vazquez,
    AUTHOR = {V\'azquez, J. L.},
     TITLE = {A strong maximum principle for some quasilinear elliptic
              equations},
   JOURNAL = {Appl. Math. Optim.},
  FJOURNAL = {Applied Mathematics and Optimization},
    VOLUME = {12},
      YEAR = {1984},
    NUMBER = {3},
     PAGES = {191--202},
      ISSN = {0095-4616,1432-0606},
   MRCLASS = {35B50 (35J60)},
  MRNUMBER = {768629},
       DOI = {10.1007/BF01449041},
}

@book {Ziemer,
    AUTHOR = {Ziemer, W. P.},
     TITLE = {Weakly differentiable functions},
    SERIES = {Graduate Texts in Mathematics},
    VOLUME = {120},
      NOTE = {Sobolev spaces and functions of bounded variation},
 PUBLISHER = {Springer-Verlag, New York},
      YEAR = {1989},
     PAGES = {xvi+308},
      ISBN = {0-387-97017-7},
   MRCLASS = {46E35},
  MRNUMBER = {1014685},
MRREVIEWER = {V.\ M.\ Gol\cprime dshte\u in},
       DOI = {10.1007/978-1-4612-1015-3},
}

@incollection {simon,
    AUTHOR = {Simon, J.},
     TITLE = {R\'egularit\'e{} de la solution d'une \'equation non
              lin\'eaire dans {${\bf R}\sp{N}$}},
 BOOKTITLE = {Journ\'ees d'{A}nalyse {N}on {L}in\'eaire ({P}roc. {C}onf.,
              {B}esan\c con, 1977)},
    SERIES = {Lecture Notes in Math.},
    VOLUME = {665},
     PAGES = {205--227},
 PUBLISHER = {Springer, Berlin},
      YEAR = {1978},
      ISBN = {3-540-08922-5},
   MRCLASS = {35D10 (35J60)},
  MRNUMBER = {519432},
MRREVIEWER = {Wolf\ von Wahl},
}

@article {Chipot-fila-shafrir,
AUTHOR = {Chipot, M. and Chlebík, M. and Fila, M. and Shafrir, I.},
TITLE = {Existence of positive solutions of a semilinear elliptic equation in {$\bold R^n_{+}$} with a nonlinear boundary condition}, 
JOURNAL = {J. Math. Anal. Appl.},
FJOURNAL = {Journal of Mathematical Analysis and Applications},
VOLUME = {223},
      YEAR = {1998},
    NUMBER = {2},
     PAGES = {429--471},
      ISSN = {0022-247X,1096-0813},
   MRCLASS = {35J65 (35B05)},
  MRNUMBER = {1629293},
MRREVIEWER = {Jianfu\ Yang},
}

@article {Do-Freire-Medeiros,
    AUTHOR = {Do Ó, J.M. and Freire, R. and Medeiros, E.S.},
     TITLE = {Liouville-type theorems and existence of solutions for quasilinear elliptic problems},
   JOURNAL = {arXiv preprint arXiv: 2607.02152},
   year={2026}
}

@article{doO_Freire_Medeiros_2026,
  author  = {Do Ó, J.M. and Freire, R. and Medeiros, E.S.},
  title   = {Weighted Sobolev trace embeddings and their applications},
 journal = {Proc. Edinb. Math. Soc.},
  fjournal = {Proceedings of the Edinburgh Mathematical Society},
  year    = {2026},
  pages   = {1--42},
  doi     = {10.1017/S0013091526101424}
}

@article{WeightedIneq,
  title={Weighted Hardy-Sobolev type inequalities with boundary terms},
  author={Do {\'O}, J.M. and Furtado, M. and Medeiros, E. and Ratzkin, J.},
  journal={arXiv e-prints},
  pages={arXiv--2602},
  year={2026}
}

@article {Damascelli,
    AUTHOR = {Damascelli, L. and Farina, A. and Sciunzi, B.
              and Valdinoci, E.},
     TITLE = {Liouville results for {$m$}-{L}aplace equations of
              {L}ane-{E}mden-{F}owler type},
   JOURNAL = {Ann. Inst. H. Poincar\'e{} C Anal. Non Lin\'eaire},
  FJOURNAL = {Annales de l'Institut Henri Poincar\'e{} C. Analyse Non
              Lin\'eaire},
    VOLUME = {26},
      YEAR = {2009},
    NUMBER = {4},
     PAGES = {1099--1119},
      ISSN = {0294-1449,1873-1430},
   MRCLASS = {35B53 (35B05 35J70 35J92 47J30)},
  MRNUMBER = {2542716},
MRREVIEWER = {Yehuda\ Pinchover},
       DOI = {10.1016/j.anihpc.2008.06.001},
       }

@article {chen,
    AUTHOR = {Chen, C.},
     TITLE = {Liouville type theorem for stable solutions of {$p$}-{L}aplace
              equation in {$\mathbb{R}^N$}},
   JOURNAL = {Appl. Math. Lett.},
  FJOURNAL = {Applied Mathematics Letters. An International Journal of Rapid
              Publication},
    VOLUME = {68},
      YEAR = {2017},
     PAGES = {62--67},
      ISSN = {0893-9659,1873-5452},
   MRCLASS = {35J92 (35B33 35B53)},
  MRNUMBER = {3614279},
       DOI = {10.1016/j.aml.2016.11.014},
}

@inproceedings {Birindelli,
    AUTHOR = {Birindelli, I. and Demengel, F.},
     TITLE = {Some {L}iouville theorems for the {$p$}-{L}aplacian},
 BOOKTITLE = {Proceedings of the 2001 {L}uminy {C}onference on {Q}uasilinear
              {E}lliptic and {P}arabolic {E}quations and {S}ystem},
    SERIES = {Electron. J. Differ. Equ. Conf.},
    VOLUME = {8},
     PAGES = {35--46},
 PUBLISHER = {Southwest Texas State Univ., San Marcos, TX},
      YEAR = {2002},
   MRCLASS = {35J60 (35D05)},
  MRNUMBER = {1990294},
}

@article {Ambrosio,
    AUTHOR = {D'Ambrosio, L. and Dipierro, S.},
     TITLE = {Hardy inequalities on {R}iemannian manifolds and applications},
   JOURNAL = {Ann. Inst. H. Poincar\'e{} C Anal. Non Lin\'eaire},
  FJOURNAL = {Annales de l'Institut Henri Poincar\'e{} C. Analyse Non
              Lin\'eaire},
    VOLUME = {31},
      YEAR = {2014},
    NUMBER = {3},
     PAGES = {449--475},
      ISSN = {0294-1449,1873-1430},
   MRCLASS = {58J05 (26D15 31C12 35A23 35J20 35R01)},
  MRNUMBER = {3208450},
MRREVIEWER = {Qu\cfac oc\ Anh\ Ng\^o},
       DOI = {10.1016/j.anihpc.2013.04.004},
}

@book {mazya,
    AUTHOR = {Maz'ya, V. G.},
     TITLE = {Sobolev spaces},
    SERIES = {Springer Series in Soviet Mathematics},
      NOTE = {Translated from the Russian by T. O. Shaposhnikova},
 PUBLISHER = {Springer-Verlag, Berlin},
      YEAR = {1985},
     PAGES = {xix+486},
      ISBN = {3-540-13589-8},
   MRCLASS = {46E35 (31C15 47F05)},
  MRNUMBER = {817985},
MRREVIEWER = {J.\ Horv\'ath},
       DOI = {10.1007/978-3-662-09922-3},
}

@article {Fi-Ma-Te1,
    AUTHOR = {Filippas, S. and Maz'ya, V. G. and Tertikas, A.},
     TITLE = {Sharp {H}ardy-{S}obolev inequalities},
   JOURNAL = {C. R. Math. Acad. Sci. Paris},
  FJOURNAL = {Comptes Rendus Math\'ematique. Acad\'emie des Sciences. Paris},
    VOLUME = {339},
      YEAR = {2004},
    NUMBER = {7},
     PAGES = {483--486},
      ISSN = {1631-073X,1778-3569},
   MRCLASS = {26D15},
  MRNUMBER = {2099546},
       DOI = {10.1016/j.crma.2004.07.023},
}

@article {Fi-Ma-Te2,
    AUTHOR = {Filippas, S. and Maz'ya, V. and Tertikas, A.},
     TITLE = {Critical {H}ardy-{S}obolev inequalities},
   JOURNAL = {J. Math. Pures Appl. (9)},
  FJOURNAL = {Journal de Math\'ematiques Pures et Appliqu\'ees. Neuvi\`eme
              S\'erie},
    VOLUME = {87},
      YEAR = {2007},
    NUMBER = {1},
     PAGES = {37--56},
      ISSN = {0021-7824},
   MRCLASS = {26D10 (35J60 46E35 58J05)},
  MRNUMBER = {2297247},
MRREVIEWER = {Ji\v r\'i\ R\'akosn\'ik},
       DOI = {10.1016/j.matpur.2006.10.007},
}

@article {lam-lu-zhang,
    AUTHOR = {Lam, N. and Lu, G. and Zhang, L.},
     TITLE = {Sharp singular {T}rudinger-{M}oser inequalities under
              different norms},
   JOURNAL = {Adv. Nonlinear Stud.},
  FJOURNAL = {Advanced Nonlinear Studies},
    VOLUME = {19},
      YEAR = {2019},
    NUMBER = {2},
     PAGES = {239--261},
      ISSN = {1536-1365,2169-0375},
   MRCLASS = {46E35 (26D10 42B35)},
  MRNUMBER = {3943303},
MRREVIEWER = {Adele\ Ferone},
       DOI = {10.1515/ans-2019-2042},
       URL = {https://doi-org.rproxy.univ-pau.fr/10.1515/ans-2019-2042},
}

\end{document}